\documentclass[a4paper,10pt]{article}
\usepackage[frenchb]{babel}
\usepackage[T1]{fontenc}
\usepackage[utf8]{inputenc}
\usepackage{amsmath}
\usepackage{array}
\usepackage{amsthm}
\usepackage{amssymb}
\input xy
\xyoption{all}

\newcommand{\A}{{\mathcal{A}}}
\newcommand{\B}{{\mathcal{B}}}

\newcommand{\F}{{\mathcal{F}}}
\newcommand{\fct}{\mathbf{Fct}}
\newcommand{\add}{\mathbf{Add}}

\newcommand{\Hh}{{\mathcal{H}}}
\newcommand{\E}{{\mathcal{E}}}

\newcommand{\Pol}{\fct}

\newcommand{\col}{{\rm colim}\,}

\title{Groupes d'extensions et foncteurs polynomiaux\\
{\em {\large D\'edi\'e \`a Teimuraz Pirashvili \`a l'occasion de son soixanti\`eme anniversaire}}}

\author{Aur\'elien DJAMENT\thanks{CNRS, Laboratoire de mathématiques Jean Leray (UMR 6629), aurelien.djament@univ-nantes.fr}}

\newtheorem{thm}{Th\'eor\`eme}[section]
\newtheorem{pr}[thm]{Proposition}

\newtheorem{lm}[thm]{Lemme}

\theoremstyle{definition}
\newtheorem{defi}[thm]{D\'efinition}

\newtheorem{hyp}[thm]{Hypoth\`ese}

\theoremstyle{remark}
\newtheorem{rem}[thm]{Remarque}
\newtheorem{ex}[thm]{Exemple}

\begin{document}

\maketitle

\begin{abstract}
Généralisant le travail de Pirashvili \cite{P-extor}, nous caractérisons les petites catégories additives $\A$ telles que l'inclusion dans la catégorie des foncteurs de $\A$ vers les groupes abéliens de la sous-catégorie pleine des foncteurs analytiques induise un isomorphisme entre groupes d'extensions.
\end{abstract}

\begin{small}
\begin{center}
 \textbf{Abstract}

\smallskip

Extending Pirashvili's work \cite{P-extor}, we characterize small additive categories $\A$ such that the inclusion in the category of functors from $\A$ to abelian groups of the full subcategory of analytic functors induces an isomorphism between extension groups. 
\end{center}
\end{small}

\medskip

\noindent
{\em Mots clefs} : catégories de foncteurs ; foncteurs polynomiaux et foncteurs analytiques ; groupes d'extensions.

\smallskip

\noindent
{\em Classification MSC 2010} : 18A25, 18G15 (18A40, 18E05, 18E15, 18G10).


\section*{Introduction}

Dans toute catégorie de foncteurs d'une catégorie additive vers une catégorie de modules\,\footnote{Ces notions peuvent être définies dans un cadre plus général (tant pour la source que pour le but), mais les résultats et méthodes qu'on présente dans cet article sont spécifiques au cas d'une source additive et d'un but abélien.}, on dispose d'une notion fondamentale d'effets croisés qui permet de définir les {\em foncteurs polynomiaux} ; les {\em foncteurs analytiques} sont les colimites de foncteurs polynomiaux (qu'on peut supposer filtrantes). Ces notions ont été introduites par Eilenberg et Mac Lane (cf. \cite{EML}) au début des années 1950 ; de fait, les groupes d'homologie singulière des espaces qui portent leur nom définissent des endofoncteurs polynomiaux $H_i(K(-,n);\mathbb{Z})$ des groupes abéliens, dont on peut estimer le degré en fonction des entiers $i$ et $n$. Depuis, l'importance de la notion de foncteur polynomial s'est confirmée en topologie algébrique (voir par exemple \cite{HLS} pour la mise en évidence d'un lien fort avec les modules sur l'algèbre de Steenrod), mais aussi en théorie des représentations (voir par exemple le survol \cite{Ku}) ou en $K$-théorie algébrique et homologie des groupes (voir par exemple l'appendice de \cite{FFSS}, ou \cite{Dja-JKT}).

L'un des succès de la théorie provient de ce que de nombreux calculs de groupes d'extensions entre foncteurs polynomiaux sont accessibles dans la catégorie de foncteurs considérée (cf. \cite{FFSS}), qui est une catégorie abélienne très régulière, avec assez d'objets projectifs et injectifs. La quasi-totalité des calculs sont menés dans la catégorie de {\em tous} les foncteurs (d'une catégorie additive raisonnable vers une catégorie de modules), même quand on ne traite que de foncteurs polynomiaux ou analytiques. Pourtant, les sous-catégories (pleines) de foncteurs polynomiaux de degré au plus $d$, où $d$ est un entier fixé, et la sous-catégorie des foncteurs analytiques sont également des catégories abéliennes fondamentales (dont on comprend plus aisément la structure globale que celle de la catégorie de tous les foncteurs) avec assez d'objets injectifs (et de projectifs dans le premier cas). La question de la comparaison des groupes d'extensions entre ces catégories, très naturelle, n'a été que peu étudiée. Le fait que, lorsque la catégorie but est celle des espaces vectoriels sur un corps de caractéristique nulle, tous ces groupes d'extensions coïncident est assez élémentaire et bien connu des experts, mais ne semble pas avoir été publié. Lorsqu'on travaille avec des espaces vectoriels sur un corps de caractéristique première, ou avec les groupes abéliens, au but, il n'en est plus de même : les extensions, dès le degré $2$ (dans le premier cas) ou $3$ (dans le second), entre deux foncteurs polynomiaux de degré $d$ font apparaître généralement des foncteurs de degré strictement supérieur à $d$. C'est essentiellement l'existence du morphisme de Frobenius qui permet un tel phénomène (cf. l'exemple classique~\ref{ex-cl} ci-après). Mais, {\em dans les cas usuels}, les groupes d'extensions entre deux foncteurs analytiques sont les mêmes dans la catégorie de tous les foncteurs et dans la catégorie des foncteurs analytiques, et les groupes d'extensions entre deux foncteurs polynomiaux coïncident dans la catégorie de tous les foncteurs et dans la catégorie des foncteurs polynomiaux de degré au plus $d$, pourvu que $d$ soit assez grand par rapport au degré des foncteurs considérés et au degré cohomologique. Le premier phénomène est bien connu, et facile à montrer (il provient de ce que toute fonction entre deux espaces vectoriels {\em finis} est polynomiale), dans les catégories de foncteurs entre espaces vectoriels (de dimension finie à la source) sur un corps fini. Le deuxième phénomène a été observé dans le cas particulier où les groupes abéliens de morphismes de la catégorie source sont sans torsion, et que l'un des arguments des groupes d'extensions est additif, par Pirashvili, dans \cite{P-extor}, qui donne également des bornes de stabilité explicites (et essentiellement optimales). Mais les autres cas semblaient largement ouverts. Signalons toutefois, encore dans le cas des foncteurs entre espaces vectoriels sur un corps fini, le calcul par Smith (non publié, mais mentionné et utilisé dans \cite{FrS}) des auto-extensions du foncteur identité dans les catégories des foncteurs polynomiaux de degrés donnés.

Dans le présent article, nous résolvons de façon générale le problème. Précisément, nous montrons que les groupes d'extensions entre foncteurs analytiques dans la catégorie de tous les foncteurs (d'une petite catégorie additive vers, disons, les groupes abéliens) ou dans la catégorie des foncteurs analytiques coïncident si et seulement si, pour tout nombre premier $p$, la torsion $p$-primaire des groupes abéliens de morphismes à la source est bornée, avec une condition d'uniformité partielle (voir le théorème~\ref{th-anag} pour l'énoncé précis). Nous donnons également des résultats de comparaison de groupes d'extensions entre catégories de foncteurs polynomiaux de degré donné et catégorie de tous les foncteurs.

\section{Énoncé des résultats}\label{sr}

Si $\A$ et $\B$ sont deux catégories, avec $\A$ (essentiellement) petite, on note $\fct(\A,\B)$ la catégorie des foncteurs de $\A$ vers $\B$. Si $\A$ et $\B$ sont des catégories additives, on note $\add(\A,\B)$ la sous-catégorie pleine de $\fct(\A,\B)$ constituée des foncteurs additifs. Si $k$ est un anneau, on note $k-\mathbf{Mod}$ la catégorie des $k$-modules à gauche ; on pose $\F(\A;k):=\fct(\A,k-\mathbf{Mod})$.

On rappelle que, si $\E$ est une catégorie abélienne, alors $\fct(\A,\E)$ est une catégorie abélienne pour toute catégorie $\A$ essentiellement petite ; si $\E$ possède assez d'injectifs ou de projectifs, ou est une catégorie de Grothendieck (on se placera généralement dans cette situation), il en est de même pour $\fct(\A,\E)$. On peut en particulier y faire de l'algèbre homologique.

Supposons que $\A$ est une catégorie additive (essentiellement petite) et $\E$ une catégorie abélienne. Pour tout entier $d$, on note $\Pol_d(\A,\E)$ la sous-catégorie pleine de $\fct(\A,\E)$ des foncteurs polynomiaux de degré au plus $d$. Cette notion remonte à Eilenberg-Mac Lane (\cite{EML}) dans le cas où source et but sont des catégories de modules ; les effets croisés et les foncteurs polynomiaux se définissent exactement de la même façon d'une catégorie additive vers une catégorie abélienne (voir par exemple \cite{P-extor} pour le cas où le but est une catégorie de modules, ou \cite{HPV}, où le cadre est encore plus général). Cette sous-catégorie est épaisse et stable par limites et colimites. On note $\F_d(\A;k):=\Pol_d(\A,k-\mathbf{Mod})$. On note $i_d : \Pol_d(\A,\E)\to\fct(\A,\E)$ le foncteur d'inclusion ; lorsque $\E$ est une catégorie de Grothendieck, ce foncteur possède un adjoint à droite qu'on note $p_d$ et un adjoint à gauche qu'on note $q_d$ (les mentions aux catégories source et but sont omises de ces notations). On considérera aussi beaucoup la catégorie des foncteurs analytiques $\Pol_\infty(\A,\E)$, c'est-à-dire la sous-catégorie pleine de $\fct(\A,\E)$ obtenue en saturant par colimites la réunion des $\Pol_d(\A,\E)$ (c'est une sous-catégorie épaisse de $\fct(\A,\E)$ --- cette propriété classique mais pas complètement formelle sera démontrée plus tard). On introduit des notations évidentes $\F_\infty(\A;k)$, $i_\infty$ et $p_\infty$ (en général, il n'y a plus de $q_\infty$ car $\Pol_\infty(\A,\E)$ n'est pas une sous-catégorie stable par limites dans $\fct(\A,\E)$).

Le but de cet article est de comparer les groupes d'extensions entre les catégories $\fct(\A,\E)$ et $\Pol_d(\A,\E)$, lorsque $\E$ est une catégorie de Grothendieck. (Dans le cas où $\E$ est une catégorie de modules sur un anneau commutatif, on pourrait tout aussi bien donner des résultats en termes de groupes de torsion.)

\begin{hyp}
 Dans toute la suite, $\A$ désigne une catégorie additive (essentiellement) petite.
\end{hyp}

Commençons par rappeler deux résultats classiques dans cette direction. Le premier d'entre eux est bien connu depuis longtemps des spécialistes des foncteurs polynomiaux, mais l'auteur du présent travail n'en connaît pas de référence écrite. Nous en rappellerons donc une démonstration ultérieurement.

\begin{thm}[Folklore]\label{th-fol}
 Supposons que $k$ est un corps de caractéristique nulle. Pour tous $d\in\mathbb{N}\cup\{\infty\}$ et tous foncteurs $F$, $G$ dans $\F_d(\A;k)$, le morphisme naturel
$${\rm Ext}^*_{\F_d(\A;k)}(F,G)\to {\rm Ext}^*_{\F(\A;k)}(i_d(F),i_d(G))$$
qu'induit le foncteur exact $i_d$ est un isomorphisme.
\end{thm}

(Dans la suite, pour alléger, nous omettrons souvent la notation des foncteurs d'inclusion $i_d$.)

\begin{rem}
 Il existe d'autres situations où la comparaison des groupes d'extensions entre catégorie de tous les foncteurs et sous-catégories de foncteurs polynomiaux de degrés donnés fonctionne de manière aussi favorable, par exemple les $\Gamma$-modules. On n'a même pas besoin, dans ce contexte, de faire une hypothèse de caractéristique nulle ou apparentée : l'équivalence de catégories de Pirashvili (\cite{PDK}) à la Dold-Kan entre $\Gamma$-modules et $\Omega$-modules (où le degré polynomial se lit de façon transparente, par l'annulation des valeurs du foncteurs sur les ensembles de cardinal strictement supérieur au degré) implique aussitôt le résultat.
\end{rem}

Dès que l'on est en caractéristique positive, le théorème précédent tombe grossièrement en défaut, comme l'illustre l'exemple très classique suivant.

\begin{ex}\label{ex-cl}
 Soient $p$ un nombre premier et $\F(p)$ la catégorie des foncteurs des $\mathbb{F}_p$-espaces vectoriels de dimension finie vers les $\mathbb{F}_p$-espaces vectoriels, on note ${\rm I}$ le foncteur d'inclusion, qui est polynomial de degré $1$. Il est classique (cf. \cite{FLS}) que ${\rm Ext}^2_{\F(p)}({\rm I},{\rm I})$ est isomorphe à $\mathbb{F}_p$, un générateur en étant donné par la classe de la suite exacte
$$0\to {\rm I}\to S^p\to\Gamma^p\to {\rm I}\to 0$$
dont la première flèche est le morphisme de Frobenius ($S^p$ désigne la $p$-ème puissance symétrique et $\Gamma^p$ la $p$-ème puissance divisée), la deuxième la norme et la troisième le morphisme de Verschiebung (dual du Frobenius).

Par ailleurs, pour $d<p$, il est classique que la sous-catégorie $\F_d(p)$ des foncteurs polynomiaux de degré (au plus) $d$ de $\F(p)$ est semi-simple (parce que les représentations des groupes symétriques $\Sigma_n$ à coefficients dans $\mathbb{F}_p$ sont semi-simples pour $n<d$ ; si l'on se restreint à $d=1$ l'assertion est de toute façon triviale), de sorte que ${\rm Ext}^2_{\F_d(p)}({\rm I},{\rm I})=0$ n'est alors pas isomorphe à ${\rm Ext}^2_{\F(p)}({\rm I},{\rm I})$. En revanche, pour $d\geq p$, le morphisme canonique ${\rm Ext}^2_{\F_d(p)}({\rm I},{\rm I})\to{\rm Ext}^2_{\F(p)}({\rm I},{\rm I})$ est un isomorphisme. En effet, c'est un toujours un monomorphisme parce que $\F_d(p)$ est une sous-catégorie {\em épaisse} de $\F(p)$, et il est clair avec ce qui précède que c'est un épimorphisme pour $d\geq p$, puisque $S^p$ et $\Gamma^p$ sont de degré $p$.
\end{ex}

En général, on ne peut donc pas s'attendre à mieux qu'à ce que le morphisme naturel
$${\rm Ext}^*_{\F_d(\A;k)}(F,G)\to {\rm Ext}^*_{\F(\A;k)}(i_d(F),i_d(G))$$
soit un isomorphisme en degré cohomologique $*$ suffisamment petit par rapport à $d$, avec une borne dépendant de conditions sur la torsion des groupes abéliens de morphismes dans la catégorie additive $\A$. (Pour $d=\infty$, on s'attend donc à un isomorphisme en tout degré cohomologique... mais nous verrons que ce n'est vrai que lorsque $\A$ est <<~assez gentille~>>.)

\smallskip

Le second théorème, beaucoup plus difficile, est dû à Pirashvili dans \cite{P-extor}. Nos résultats permettront de le généraliser.

\begin{thm}[Pirashvili]\label{th-pira}
 Soient $d,i\in\mathbb{N}$, $F$ et $G$ deux objets de $\F_d(\A;\mathbb{Z})$. Supposons que :
\begin{enumerate}
 \item les groupes abéliens $\A(a,b)$ sont tous sans torsion ;
\item $F$ ou $G$ est additif ;
\item $i\leq 2d$.
\end{enumerate}

Alors le morphisme canonique
$${\rm Ext}^i_{\F_d(\A;\mathbb{Z})}(F,G)\to {\rm Ext}^i_{\F(\A;\mathbb{Z})}(F,G)$$
est un isomorphisme.
\end{thm}

(Il n'est pas difficile de reprendre les arguments de Pirashvili pour inclure aussi le cas $d=\infty$ dans l'énoncé.)

La démonstration de ce théorème donnée dans \cite{P-extor} repose sur un examen minutieux de la construction cubique de Mac Lane et de la filtration polynomiale (i.e. par les puissances de l'idéal d'augmentation) de l'anneau de groupe $\mathbb{Z}[V]$, où $V$ est un groupe abélien sans torsion (un point fondamental est un argument d'idéal {\em quasi-régulier} --- notion introduite par Quillen dans ses travaux sur l'homologie des algèbres --- qui repose fortement sur l'absence de torsion).

\begin{rem}
 Pirashvili donne également un énoncé analogue pour les groupes de torsion (qui se démontre exactement de la même façon). On laissera au lecteur le soin d'écrire les traductions en termes de Tor des résultats du présent article.
\end{rem}

\medskip

Donnons maintenant nos résultats. On commence par se concentrer sur le cas où la catégorie but est la catégorie des espaces vectoriels sur le corps $\mathbb{F}_p=\mathbb{Z}/p$, où $p$ est un nombre premier fixé.

\begin{thm}\label{th1}
 Supposons qu'il existe un entier naturel $r$ tel que, pour tous objets $a$ et $b$ de $\A$, la torsion $p$-primaire du groupe abélien $\A(a,b)$ soit bornée par $p^r$.

Soient $d,n,i\in\mathbb{N}$ tels que $d\leq n$, $F$ un foncteur de $\F_d(\A;\mathbb{F}_p)$, $G$ un foncteur de $\F_n(\A;\mathbb{F}_p)$. Alors le morphisme canonique
$${\rm Ext}^i_{\F_n(\A;\mathbb{F}_p)}(F,G)\to {\rm Ext}^i_{\F(\A;\mathbb{F}_p)}(F,G)$$
est un isomorphisme pour $i\leq \big[\frac{n-d+1}{p^r}\big]$ et un monomorphisme pour $i=\big[\frac{n-d+1}{p^r}\big]+1$.

Le même résultat vaut si l'on échange $F$ et $G$ dans les groupes d'extensions.
\end{thm}

(Dans cet énoncé, les crochets désignent la partie entière.)

\begin{thm}\label{th2}
 Les assertions suivantes sont équivalentes.
\begin{enumerate}
 \item Pour tout objet $a$ de $\A$, il existe $r\in\mathbb{N}$ tel que, pour tout objet $b$ de $\A$, la torsion $p$-primaire du groupe abélien $\A(a,b)$ soit bornée par $p^r$ ;
\item pour tous objets $F$ et $G$ de $\F_\infty(\A;\mathbb{F}_p)$, le morphisme canonique
$${\rm Ext}^*_{\F_\infty(\A;\mathbb{F}_p)}(F,G)\to {\rm Ext}^*_{\F(\A;\mathbb{F}_p)}(F,G)$$
est un isomorphisme ;
\item pour tous objets $F$ et $G$ de $\F_1(\A;\mathbb{F}_p)$, le morphisme canonique
$${\rm Ext}^2_{\F_\infty(\A;\mathbb{F}_p)}(F,G)\to {\rm Ext}^2_{\F(\A;\mathbb{F}_p)}(F,G)$$
est un isomorphisme.
\end{enumerate}
\end{thm}

\begin{pr}\label{pr3}
 Soient $d\in\mathbb{N}$, $F$ et $G$ des foncteurs de $\F_d(\A;\mathbb{F}_p)$. Le morphisme canonique
$$\underset{r\geq d}{\col} {\rm Ext}^*_{\F_r(\A;\mathbb{F}_p)}(F,G)\to \underset{p^t>d}{\col}{\rm Ext}^*_{\F(\A/p^t;\mathbb{F}_p)}(F,G)$$
est un isomorphisme.
\end{pr}

Cet énoncé mérite quelques précisions : pour $n\in\mathbb{N}$, on note $\A/n$ la catégorie additive ayant les mêmes objets que $\A$ et dont les morphismes sont donnés par $(\A/n)(a,b)=\A(a,b)/n$, la composition étant induite par celle de $\A$, de sorte qu'on dispose d'un foncteur canonique additif, plein et essentiellement surjectif $\A\to\A/n$. On dispose donc d'un diagramme commutatif de foncteurs d'inclusion
$$\xymatrix{\F_d(\A/p^t;\mathbb{F}_p)\ar[r]\ar[d] & \F(\A/p^t;\mathbb{F}_p)\ar[d] \\
\F_d(\A;\mathbb{F}_p)\ar[r] & \F(\A;\mathbb{F}_p)
}$$
dans lequel on voit sans trop de peine (nous donnerons les détails de l'argument dans la section~\ref{so}) que la flèche verticale de gauche est une égalité si $p^t>d$, ce qui procure un foncteur (exact et pleinement fidèle) $\F_d(\A;\mathbb{F}_p)\to\F(\A/p^t;\mathbb{F}_p)$ sous cette condition. Le morphisme canonique de la proposition est induit par ces foncteurs.

\smallskip

Notre dernier résultat sur $\F(\A;\mathbb{F}_p)$ constitue le c\oe ur de la démonstration de l'implication {\em 3.}$\Rightarrow${\em 1.} du théorème~\ref{th1}.

\begin{pr}\label{pr-ext2}
 Soient $A$ un foncteur additif de $\F(\A;\mathbb{F}_p)$ et $T : \A^{op}\to\mathbf{Ab}$ un foncteur additif projectif de la forme $\underset{i}{\bigoplus}\A(-,a_i)$, où $(a_i)$ est une collection d'objets de $\A$. On note $B:={\rm Hom}_\mathbb{Z}(T,\mathbb{F}_p)$, c'est donc un foncteur additif de $\F(\A;\mathbb{F}_p)$.

On dispose d'un diagramme commutatif
$$\xymatrix{\underset{d\in\mathbb{N}^*}{\col} {\rm Ext}^2_{\F_d(\A;\mathbb{F}_p)}(A,B)\ar[d]\ar[r]^-\simeq & \underset{i\in\mathbb{N}}{\col} {\rm Hom}_{\F(\A;\mathbb{F}_p)}\big(A,{\rm Ext}^1_\mathbb{Z}(-,\mathbb{Z}/p)\circ (T/p^i)\big)\ar[d] \\
{\rm Ext}^2_{\F_\infty(\A;\mathbb{F}_p)}(A,B)\ar[d]\ar[r]^-\simeq & {\rm Hom}_{\F(\A;\mathbb{F}_p)}\big(A,\underset{i\in\mathbb{N}}{\col}{\rm Ext}^1_\mathbb{Z}(-,\mathbb{Z}/p)\circ (T/p^i)\big)\ar[d] \\
{\rm Ext}^2_{\F(\A;\mathbb{F}_p)}(A,B)\ar[r]^-\simeq & {\rm Hom}_{\F(\A;\mathbb{F}_p)}\big(A,{\rm Ext}^1_\mathbb{Z}(-,\mathbb{Z}/p)\circ T\big)
}$$
dont les flèches verticales sont induites par les foncteurs d'inclusion.
\end{pr}

\smallskip

On termine en donnant deux résultats de comparaison dans lesquels la catégorie but des foncteurs est une catégorie de Grothendieck quelconque. Comme on le verra dans la section~\ref{so}, ils se déduisent des précédents par un argument élémentaire de changement de base (on pourrait bien sûr en donner quelques autres conséquences).

\begin{hyp}
 Dans la suite, $\E$ désigne une catégorie de Grothendieck.
\end{hyp}

\begin{thm}\label{th-db}
 Supposons qu'il existe un entier $N>0$ tel que, pour tous objets $a$ et $b$ de $\A$, le sous-groupe de torsion du groupe abélien $\A(a,b)$ soit annulé par $N$.

Soient $d,n,i\in\mathbb{N}$ tels que $d\leq n$, $F$ un foncteur de $\Pol_d(\A;\E)$, $G$ un foncteur de $\Pol_n(\A;\E)$. Alors le morphisme canonique
$${\rm Ext}^i_{\Pol_n(\A,\E)}(F,G)\to {\rm Ext}^i_{\fct(\A,\E)}(F,G)$$
est un isomorphisme pour $i\leq \big[\frac{n-d+1}{N}\big]$ et un monomorphisme pour $i=\big[\frac{n-d+1}{N}\big]+1$.

Le même résultat vaut si l'on échange $F$ et $G$ dans les groupes d'extensions.
\end{thm}

\begin{thm}\label{th-anag} Supposons que la condition suivante soit satisfaite :
 
 Pour tout objet $a$ de $\A$ et tout nombre premier $p$, il existe $r\in\mathbb{N}$ tel que, pour tout objet $b$ de $\A$, la torsion $p$-primaire du groupe abélien $\A(a,b)$ soit bornée par $p^r$.

Alors, pour tous objets $F$ et $G$ de $\F_\infty(\A;\E)$, le morphisme canonique
$${\rm Ext}^*_{\Pol_\infty(\A,\E)}(F,G)\to {\rm Ext}^*_{\fct(\A,\E)}(F,G)$$
est un isomorphisme.

La réciproque est vraie si $\E$ est la catégorie des groupes abéliens ; dans ce cas, si la condition précédente n'est pas réalisée, on peut trouver des foncteurs additifs $F$ et $G$ tels que le morphisme précédent ne soit pas un isomorphisme en degré cohomologique au plus $3$.
\end{thm}

\section{Outils et premières réductions}\label{so}

\subsection{Adjonctions et suites spectrales de Grothendieck}
Pour tout entier naturel $d$, le foncteur $i_d : \Pol_d(\A,\E)\to\fct(\A,\E)$ possède un adjoint à droite $p_d$ et un adjoint à droite $q_d$ car $i_d$ commute aux limites et colimites (les catégories sont de Grothendieck). Explicitement, $p_d(F)$ (resp. $q_d(F)$) est le plus grand sous-foncteur (resp. quotient) de $F$ polynomial de degré au plus $d$ ; on peut l'écrire aussi comme noyau (resp. conoyau) d'un morphisme naturel utilisant une variante des effets croisés. On note aussi que la précomposition par un foncteur additif ou la postcomposition par un foncteur exact commutent aux foncteurs $i_d$, $p_d$ et $q_d$.

Si $\E$ est la catégorie des modules sur un anneau commutatif $k$, on dispose également des utiles formules suivantes pour les dérivés de ces foncteurs :
\begin{equation}\label{eqa1}
\mathbf{R}^\bullet p_d(F)(a)={\rm Ext}^\bullet_{\F(\A;k)}\big(q_d(k[\A(a,-)]),F\big)
\end{equation}

\begin{equation}\label{eqa2}
\mathbf{L}_\bullet q_d(F)(a)={\rm Tor}_\bullet^\A\big(q_d(k[\A(-,a)]),F\big).
\end{equation}

On peut en fait utiliser ces mêmes formules lorsque $\E$ est une catégorie de Grothendieck quelconque, en remarquant qu'on dispose de bifoncteurs
$${\rm Hom} : \F(\A;\mathbb{Z})^{op}\times\fct(\A,\E)\to\E$$
et
$$\underset{\A}{\otimes} : \F(\A^{op};\mathbb{Z})\times\fct(\A,\E)\to\E$$
qu'on peut dériver de la même façon que les Hom et produits tensoriels usuels, obtenant :

\begin{equation}\label{eqa3}
\mathbf{R}^\bullet p_d(F)(a)={\rm Ext}^\bullet\big(q_d(\mathbb{Z}[\A(a,-)]),F\big)
\end{equation}
et
\begin{equation}\label{eqa4}
\mathbf{L}_\bullet q_d(F)(a)={\rm Tor}_\bullet^\A\big(q_d(\mathbb{Z}[\A(-,a)]),F\big)
\end{equation}
pour $F$ dans $\fct(\A,\E)$.

(Toutes ces formules ne sont que des variations du lemme de Yoneda.)

\smallskip

Le foncteur $i_\infty : \Pol_\infty(\A,\E)\to\fct(\A,\E)$ commute également aux colimites (mais pas aux limites) et possède donc un adjoint à droite $p_\infty$. Il s'obtient en prenant la colimite sur $d\in\mathbb{N}$ des $p_d$ (plus précisément, $i_\infty p_\infty\simeq\underset{d\in\mathbb{N}}{\col}i_d p_d$) ; il en est de même pour ses dérivés à droite.

\smallskip

Comme $i_d$ est un foncteur exact, on dispose de suites spectrales de Grothendieck
\begin{equation}\label{gr1}
 E_2^{i,j}={\rm Ext}^i_{\Pol_d(\A,\E)}(G,\mathbf{R}^j p_d(F))\Rightarrow {\rm Ext}^{i+j}_{\fct(\A,\E)}(i_d(G),F)
\end{equation}
(pour $d\in\mathbb{N}\cup\{\infty\}$) et
\begin{equation}\label{gr2}
E^2_{i,j}={\rm Ext}^i_{\Pol_d(\A,\E)}(\mathbf{L}_j q_d(F),G)\Rightarrow {\rm Ext}^{i+j}_{\fct(\A,\E)}(F,i_d(G))
\end{equation}
(pour $d\in\mathbb{N}$).

Pour démontrer nos résultats, nous chercherons donc des zones d'annulation pour $\mathbf{R}^j p_d(F)$ ou $\mathbf{L}_j q_d(F)$ lorsque $j$ est non nul et assez petit par rapport à $d$.

\subsection{Changement de catégorie but}

Nous expliquons ici comment déduire le théorème~\ref{th-db} des théorèmes~\ref{th1} et~\ref{th-fol}, ainsi que le théorème~\ref{th-anag} du théorème~\ref{th2}. On commence par traiter le cas de la catégorie but <<~universelle~>> $\E=\mathbf{Ab}$.

Tous les arguments étant similaires et aisés, on se bornera à montrer le suivant :
si
$${\rm Ext}^*_{\F_\infty(A;k)}(F,G)\to {\rm Ext}^*_{\F(A;k)}(F,G)$$
est un isomorphisme pour tous foncteurs analytiques $F$ et $G$ et tout corps premier $k$, alors il en est de même pour $k=\mathbb{Z}$.

En effet, notons $u : \mathbb{F}_p-\mathbf{Mod}\to\mathbb{Z}-\mathbf{Mod}$ le foncteur d'oubli, $\sigma=-\otimes\mathbb{F}_p$ son adjoint à gauche et $\tau={\rm Hom}\,(\mathbb{F}_p,-)$ son adjoint à droite. Ces adjonctions induisent par postcomposition (qu'on désigne par une étoile en indice) des adjonctions qui se dérivent en des suites spectrales
$$E_2^{i,j}={\rm Ext}^i(({\bf L}^j\sigma_*)(F),G)\Rightarrow {\rm Ext}^{i+j}(F,u_*(G))$$
et
$$E_2^{i,j}={\rm Ext}^i(F,({\bf R}^j\tau_*)(G))\Rightarrow {\rm Ext}^{i+j}(u_*(F),G)$$
avec ${\bf L}^1\sigma_*=\tau_*$ et ${\bf R}^1\tau_*=\sigma_*$ et les dérivés de $\sigma_*$ et $\tau_*$ nuls en degré strictement supérieur à $1$. Tout cela est aussi bien valable entre $\F(\A;\mathbb{Z})$ et $\F(\A;\mathbb{F}_p)$ qu'entre $\F_d(\A;\mathbb{Z})$ et $\F_d(\A;\mathbb{F}_p)$, pour $d\in\mathbb{N}\cup\{\infty\}$. On en déduit que ${\rm Ext}^*_{\F_\infty(\A;\mathbb{Z})}(F,G)\to {\rm Ext}^*_{\F(\A;\mathbb{Z})}(F,G)$ est un isomorphisme si $F$ ou $G$ appartient à l'image essentielle de $u_*$. Pour le cas général on utilise les suites exactes
$$0\to_p\negmedspace F\to F\xrightarrow{.p}F\to F/p\to 0$$
pour obtenir que noyau et conoyau de ${\rm Ext}^*_{\F_\infty(\A;\mathbb{Z})}(F,G)\to {\rm Ext}^*_{\F(\A;\mathbb{Z})}(F,G)$ sont $p$-divisibles pour tout nombre premier $p$ ; comme ce sont aussi des groupes de torsion (à cause du théorème~\ref{th-fol}), ils sont nuls.

\smallskip

Montrons maintenant comment passer de la catégorie but $\mathbf{Ab}$ (ou des différentes catégories d'espaces vectoriels sur les corps premiers) à une catégorie de Grothendieck arbitraire $\E$. Là encore, on se contente du théorème~\ref{th-anag}, le théorème~\ref{th-db} se déduisant de la manière analogue (et même beaucoup plus directe, en utilisant la formule~(\ref{eqa4})) du cas des groupes abéliens.

On remarque d'abord qu'il suffit de le faire lorsque la catégorie de Grothendieck $\E$ est $k$-linéaire, où $k$ est un corps premier (c'est-à-dire que les groupes abéliens $\E(a,b)$ sont tous des $k$-espaces vectoriels). En effet, pour tout corps premier $k$, on dispose d'une catégorie de Grothendieck $k$-linéarisée $\E_k$, à savoir la sous-catégorie pleine de $\E$ des objets dont l'anneau des endomorphismes est une $k$-algèbre. Le foncteur d'inclusion $\E_k\to\E$ possède des adjoints de chaque côté qui se comportent exactement comme dans le cas précédent ($\E=\mathbf{Ab}$), de sorte que le même raisonnement montre que le résultat pour chaque $\E_k$ implique celui pour $\E$.

Ceci établi, on suppose donc $\E$ $k$-linéaire. On montre d'abord que le morphisme canonique
\begin{equation}\label{tech}
 {\rm Ext}^*_{\Pol_\infty(\A,\E)}(F,G)\to {\rm Ext}^*_{\fct(\A,\E)}(F,G)
\end{equation}
est un isomorphisme pour tous foncteurs analytiques $F$ et $G$ (sous l'hypothèse que c'est vrai pour $\E=k-\mathbf{Mod}$) lorsque $F$ est de la forme $X\otimes T$, où $X$ est un foncteur analytique de $\F(\A;k)$ et $T$ un objet de $\E$. Ici, $X\otimes T$ est à comprendre comme la composée de $X$ et du foncteur $-\otimes T : k-\mathbf{Mod}\to\E$. Du fait que le produit tensoriel sur un corps est exact en chaque variable, l'isomorphisme naturel d'adjonction
$${\rm Hom}_{\fct(\A,\E)}(X\otimes T,G)\simeq {\rm Hom}_{\F(\A;\mathbb{Z})}(X,{\rm Hom}_\E(T,-)\circ G)$$
se dérive en une suite spectrale
$$E_2^{i,j}={\rm Ext}^i_{\F(\A;\mathbb{Z})}(X,{\rm Ext}^j_\E(T,-)\circ G)\Rightarrow {\rm Ext}^{i+j}_{\fct(\A,\E)}(X\otimes T,G).$$
Le même résultat vaut en remplaçant $\F(\A;k)$ et $\fct(\A,\E)$ par leurs sous-catégories de foncteurs analytiques, de sorte que le morphisme (\ref{tech}) est bien un isomorphisme pour $F=X\otimes T$ avec $X$. Le cas général s'en déduit parce que la sous-catégorie $\Pol_\infty(\A,\E)$ de $\fct(\A,\E)$ est engendrée par les foncteurs analytiques de cette forme. En effet, elle est clairement engendrée par les foncteurs
$$q_d(k[\A(a,-)]\otimes T)\simeq q_d(k[\A(a,-)])\otimes T$$
où $d$ parcourt $\mathbb{N}$ et $a$ les objets de $\A$, d'où notre assertion.

\subsection{Remarques formelles liées à des propriétés de finitude}

Le morphisme canonique
$${\rm Ext}^i_{\Pol_d(\A,\E)}(F,G)\to {\rm Ext}^i_{\fct(\A,\E)}(F,G)$$
est toujours un isomorphisme pour $i\leq 1$ et un monomorphisme pour $i=2$, car $\Pol_d(\A,\E)$ est une sous-catégorie {\em épaisse} de $\fct(\A,\E)$. Cette propriété est évidente pour $d$ fini (les foncteurs polynomiaux sont définis par l'annulation de certains effets croisés, or les focnteurs d'effets croisés sont exacts --- ils commutent même à toutes les limites et colimites) ; pour $d=\infty$, elle est un peu moins claire (la saturation par colimites d'une sous-catégorie épaisse d'une catégorie de Grothendieck n'est pas toujours épaisse).

Lorsque $\E$ est la catégorie des groupes abéliens (ou une catégorie de modules), la stabilité par extensions des foncteurs polynomiaux provient formellement de ce que $\F_\infty(\A;\mathbb{Z})$ est engendrée par les foncteurs $q_d (\mathbb{Z}[\A(a,-)])$, qui sont de présentation finie dans $\F(\A;\mathbb{Z})$ (de sorte que ${\rm Ext}^1(q_d (\mathbb{Z}[\A(a,-)]),-)$ commute aux colimites filtrantes de monomorphismes). Cet argument montre en fait aussi le résultat lorsque $\E$ est quelconque, à cause de la formule (\ref{eqa3}). En effet, le caractère épais de $\Pol_d(\A,\E)$ est équivalent à l'annulation du foncteur $\mathbf{R}^1 p_d$ sur $\Pol_d(\A,\E)$. Comme on connaît le résultat pour $d$ fini et que $\mathbf{R}^1 p_\infty\simeq\underset{d\in\mathbb{N}}{\col}\mathbf{R}^1 p_d$, il suffit de vérifier que les $\mathbf{R}^1 p_d$ commutent, pour $d$ fini, aux colimites filtrantes de monomorphismes, ce qui est précisément une conséquence formelle du résultat de présentation finie précédent et de (\ref{eqa3}).

\begin{rem}\label{rcfm}
 En revanche, pour $i>1$, le foncteur $\mathbf{R}^i p_d$ ne commute pas toujours aux colimites filtrantes de monomorphismes (quel que soit $d>0$).
\end{rem}

Penchons-nous sur le morphisme canonique
$$\underset{r\geq d}{\col} {\rm Ext}^*_{\F_r(\A;k)}(F,G)\to {\rm Ext}^*_{\F_\infty(\A;k)}(F,G)\,,$$
où $F$ et $G$ sont des foncteurs de $\F_d(\A;k)$ (on se limite ici à une catégorie de modules au but par simple commodité).

Pour des raisons formelles, on peut voir facilement que c'est un isomorphisme lorsque les deux conditions suivantes sont satisfaites :
\begin{enumerate}
 \item $F$ est de type fini ;
\item la catégorie $\F_\infty(\A;k)$ est localement noethérienne (ce qui équivaut à dire que chaque catégorie $\F_d(\A;k)$, pour $d$ fini, est localement noethérienne).
\end{enumerate}

Cela provient de ce que, sous la deuxième hypothèse, on peut former une résolution injective $I^\bullet_\infty$ de $G$ dans $\F_\infty(\A;k)$ en prenant la colimite sur les entiers $r\geq d$ de résolutions injectives $I^\bullet_r$ de $G$ dans $\F_r(\A;k)$ (le caractère localement noethérien de $\F_\infty(\A;k)$ implique qu'un objet $J$ de celle-ci est injectif si ${\rm Ext}^1_{\F_\infty(\A;k)}(F,J)=0$ pour tout $F$ noethérien, mais un foncteur analytique noethérien est polynomial, et nécessairement de présentation finie dans $\F_\infty(\A;k)$, en utilisant encore le caractère localement noethérien de cette catégorie, ce qui permet de conclure par un argument de colimite filtrante).

Mais il existe des petites catégories additives a priori très <<~raisonnables~>> telles que $\F_1(\A;\mathbb{Z})$ (ou $\F_1(\A;\mathbb{F}_p)$) ne soit pas localement noethérienne, par exemple la sous-catégorie pleine de $\mathbf{Ab}$ constituée des groupes abéliens finis : le foncteur associant à un groupe son sous-groupe de torsion $p$-primaire n'est pas de type fini, alors que c'est un sous-foncteur du foncteur d'inclusion, qui est de type fini.

\subsection{Filtration polynomiale sur $\mathbb{F}_p[-]$}

On se fixe ici un nombre premier $p$.

Pour $d\in\mathbb{N}$, on note $Q^d_{(p)}$ le foncteur $q_d\mathbb{F}_p[-] : \mathbf{Ab}\to\mathbb{F}_p-\mathbf{Mod}$ et $S^d_{(p)}$ le noyau de la projection $Q^d_{(p)}\twoheadrightarrow Q^{d-1}_{(p)}$. On a classiquement :
\begin{enumerate}
 \item $S^*_{(p)}$ est un foncteur exponentiel de Hopf (i.e. on dispose d'isomorphismes naturels d'espaces vectoriels gradués $S^*_{(p)}(U\oplus V)\simeq S^*_{(p)}(U)\otimes S^*_{(p)}(V)$ vérifiant les conditions monoïdales usuelles). Ainsi, $S^*_{(p)}(V)$ est naturellement une algèbre graduée (c'est l'algèbre graduée associée à la filtration de l'algèbre de groupe $\mathbb{F}_p[V]$ par les puissances de son idéal d'augmentation) ;
\item cette algèbre graduée est naturellement isomorphe au quotient de l'algèbre (graduée) symétrique sur le $\mathbb{F}_p$-espace vectoriel $\mathbb{F}_p\otimes V$ par l'idéal (homogène) engendré par les éléments $v^{p^n}$, où $v\in V$ est tel que $p^n.v=0$. En particulier, $S^*_{(p)}(V)$ est naturellement isomorphe à l'algèbre symétrique sur $\mathbb{F}_p\otimes V$ si $V$ n'a pas de $p$-torsion, et $S^d_{(p)}(\mathbb{Z}/p^r)$ est de dimension $1$ pour $d<p^r$ et nul sinon (on pourra par exemple se référer à \cite{Passi}, chap.~VIII) ;
\item la description duale suivante (où l'exposant $\vee$ indiquant la dualité des $\mathbb{F}_p$-espaces vectoriels) nous sera également utile : $Q_d(V)^\vee\hookrightarrow\mathbb{F}_p[V]^\vee\simeq Fct(V,\mathbb{F}_p)$ (espace vectoriel des fonctions ensemblistes $V\to\mathbb{F}_p$) s'identifie au sous-espace $Pol_d(V,\mathbb{F}_p)$ des fonctions polynomiales de degré au plus $d$.
\end{enumerate}

Un fait élémentaire mais fondamental dans tout ce travail est le suivant :
\begin{enumerate}
 \item pour $p^i>d$, l'inclusion $Pol_d(V/p^i,\mathbb{F}_p)\hookrightarrow Pol_d(V,\mathbb{F}_p)$ qu'induit la projection $V\twoheadrightarrow V/p^i$ est une égalité. Comme corollaire, sous la même hypothèse, le foncteur d'inclusion $\F_d(\A/p^i;\mathbb{F}_p)\to\F_d(\A;\mathbb{F}_p)$ est une égalité (nous l'avons utilisé pour définir la flèche de la proposition~\ref{pr3}) ;
\item en particulier, si l'on note $Pol(U,V)$ (où $U$ et $V$ sont deux groupes abéliens) le groupe abélien des fonctions polynomiales $U\to V$ (i.e. la colimite sur $d$ des $Pol_d(U,V)$), le morphisme canonique
$$\underset{i\in\mathbb{N}}{\col} Pol(V/p^i,\mathbb{F}_p)\to Pol(V,\mathbb{F}_p)$$
est un isomorphisme ;
\item comme l'inclusion $Pol(U,\mathbb{F}_p)\hookrightarrow Fct(U,\mathbb{F}_p)$ est une égalité si $U$ est un $p$-groupe abélien, on en déduit une inclusion naturelle
\begin{equation}\label{eqp}
Pol(V,\mathbb{F}_p)\hookrightarrow\underset{i\in\mathbb{N}}{\col} Fct(V/p^i,\mathbb{F}_p)
\end{equation}
qui est un isomorphisme si $V$ est de type fini.
\end{enumerate}

\subsection{Complexes de type Koszul}

La structure exponentielle graduée sur $S^*_{(p)}$ ($p$ étant toujours un nombre premier fixé) permet de définir, pour tout $n\in\mathbb{N}$, un complexe $K^n_*$ :
$$0\to\mathbb{F}_p\otimes\Lambda^n\to S^1_{(p)}\otimes\Lambda^{n-1}\to S^2_{(p)}\otimes\Lambda^{n-2}\to\dots\to S^{n-1}_{(p)}\otimes\Lambda^1\to S^n_{(p)}\to 0$$
de foncteurs $\mathbf{Ab}\to\mathbb{F}_p-\mathbf{Mod}$. Ce complexe est un quotient du complexe de Koszul usuel, où apparaissent les puissances symétriques.

Explicitement, la différentielle $\Lambda^i\otimes S^{n-i}_{(p)}\to\Lambda^{i-1}\otimes S^{n-i+1}_{(p)}$ (le terme de gauche devant être tensorisé par $\mathbb{F}_p$ si $i=n$) est composée de la flèche $\Lambda^i\otimes S^{n-i}_{(p)}\to\Lambda^{i-1}\otimes\Lambda^1\otimes S^{n-i}_{(p)}=\Lambda^{i-1}\otimes S^1_{(p)}\otimes S^{n-i}_{(p)}$ induite par le coproduit $\Lambda^i\to\Lambda^{i-1}\otimes\Lambda^1$ de l'algèbre extérieure et de la flèche $\Lambda^{i-1}\otimes S^1_{(p)}\otimes S^{n-i}_{(p)}\to\Lambda^{i-1}\otimes S^{n-i+1}_{(p)}$ induite par le produit $S^1_{(p)}\otimes S^{n-i}_{(p)}\to S^{n-i+1}_{(p)}$ de l'algèbre $S^*_{(p)}$.

Notons $H_*(n)$ l'homologie de $K^n_*$, que l'on gradue de sorte que $S^{n-i}_{(p)}\otimes\Lambda^i$
se trouve en degré $i$.

\begin{pr}\label{kosz}
 Si la torsion $p$-primaire du groupe abélien $V$ est bornée par $p^r$, alors $H_i(n)(V)=0$ pour $n>p^r i$.
\end{pr}

\begin{proof}
 Les foncteurs $\Lambda^i$ et $S^i_{(p)}$, donc aussi $H_i(n)$, commutent aux colimites filtrantes, de sorte qu'il suffit de montrer le résultat lorsque $V$ est de type fini. De plus, le foncteur bigradué $H_*(\bullet)$ est exponentiel (on dispose d'iso\-morphismes naturels $H_*(\bullet)(U\oplus V)\simeq H_*(\bullet)(U)\otimes H_*(\bullet)(V)$ de $\mathbb{F}_p$-espaces vectoriels bigradués), car on dispose d'isomorphismes naturels $K^\bullet_*(U\oplus V)\simeq K^\bullet_*(U)\otimes K^\bullet_*(V)$ compatibles aux différentielles (cela découle de la même propriété, aisée et bien connue, du complexe de Koszul usuel).

On en déduit qu'il suffit d'établir la propriété lorsque $V$ est un groupe cyclique. Dans ce cas, on a $\Lambda^i(V)=0$ pour $i>1$, de sorte que le complexe se réduit au morphisme canonique (produit) $S^{n-1}_{(p)}(V)\otimes V\to S^n_{(p)}(V)$. Celui-ci est surjectif (c'est-à-dire que $H_0(n)(V)$ est nul) sauf pour $n=0$, et injectif (c'est-à-dire que $H_1(n)(V)$ est nul) sauf si $S^n_{(p)}(V)$ est nul et $S^{n-1}_{(p)}(V)$ non nul (auquel cas il est isomorphe à $\mathbb{F}_p$). Cela arrive seulement si $V$ est d'ordre fini, disons $p^t$, et $n=p^t$. Comme $t\leq r$ par hypothèse, on a dans tous les cas $H_i(n)(V)=0$ pour $n>p^r i$, comme souhaité.
\end{proof}

\subsection{Un résultat d'annulation à la Pirashvili}

On commence par rappeler le résultat classique suivant, établi par Pirashvili dans \cite{P-add}) ; son utilité est constante dans les considérations cohomologiques sur les catégories de foncteurs d'une catégorie additive vers une catégorie de modules. 

\begin{lm}[Pirashvili]\label{lmpira}
 Si $F$ est un foncteur additif de $\F(\A;k)$ et $A$ et $B$ sont deux foncteurs réduits (c'est-à-dire nuls en $0$), ${\rm Ext}^*_{\F(A;k)}(F,A\otimes B)$ et ${\rm Ext}^*_{\F(A;k)}(A\otimes B,F)$ sont nuls. Plus généralement, si $F$ est un objet de $\F_d(A;k)$ et $A_0,A_1,\dots, A_d$ sont $d+1$ foncteurs réduits, ${\rm Ext}^*_{\F(A;k)}(F,A_0\otimes\dots\otimes A_d)$ et ${\rm Ext}^*_{\F(A;k)}(A_0\otimes\dots\otimes A_d,F)$ sont nuls.
\end{lm}

Ce lemme nous suffirait si nous restreignions nos énoncés au cas où l'un des deux arguments des groupes d'extensions est de degré $1$. Pour traiter de foncteurs polynomiaux de degré arbitraire, nous aurons besoin de la généralisation suivante.

\begin{lm}\label{pira-gl}
 Soient $d$, $i$, $j$ des entiers naturels et $A$, $B$ des foncteurs réduits de $\F(\A;\mathbb{F}_p)$. On suppose que ${\rm Ext}^p_{\F(\A;\mathbb{F}_p)}(A,F)$ (resp. ${\rm Ext}^q_{\F(\A;\mathbb{F}_p)}(B,F)$) est nul lorsque $F$ appartient à $\F_d(\A;\mathbb{F}_p)$ et que $p<i$ (resp. $q<j$). Alors ${\rm Ext}^n_{\F(\A;\mathbb{F}_p)}(A\otimes B,T)=0$ pour $T$ dans $\F_{d+1}(\A;\mathbb{F}_p)$ et $n<i+j$. 
\end{lm}

\begin{proof}
 En utilisant l'adjonction entre diagonale $\A\to\A\times\A$ et somme $\A\times\A\to\A$ et le fait que $A$ et $B$ sont réduits, on obtient classiquement
$${\rm Ext}^*_{\F(\A;\mathbb{F}_p)}(A\otimes B,T)\simeq {\rm Ext}^*_{\F(\A\times\A;\mathbb{F}_p)}(A\boxtimes B,cr_2(T)),$$
où $\boxtimes$ désigne le produit tensoriel extérieur.

Par ailleurs, pour tout bifoncteur $X$, on dispose d'une suite spectrale
$$E_2^{p,q}={\rm Ext}^p_{\F(\A;\mathbb{F}_p)}(A,a\mapsto{\rm Ext}^q_{\F(\A;\mathbb{F}_p)}(B,X(a,-)))\Rightarrow {\rm Ext}^{p+q}_{\F(\A\times\A;\mathbb{F}_p)}(A\boxtimes B,X).$$
Pour $X=cr_2(T)$ avec $T$ dans $\F_{d+1}(\A)$, $X(a,-)$ appartient à $\F_d(\A;\mathbb{F}_p)$ pour tout objet $a$ de $\A$, de même que les foncteurs $a\mapsto {\rm Ext}^q_{\F(\A;\mathbb{F}_p)}(B,X(a,-))$. Par conséquent, $E_2^{p,q}=0$ si $p<a$ (par l'hypothèse sur $A$) ou $q<b$ (par l'hypothèse sur $B$), donc en particulier si $p+q<a+b$, d'où le lemme.
\end{proof}

\subsection{Stabilisation à La Dold-Puppe généralisée}\label{pdp}

Pour tout $n\in\mathbb{N}^*$ et tout foncteur $F$ de $\fct(\A,\E)$, on dispose d'un complexe de chaînes naturel $D^{(n)}_*(F)$ de foncteurs de $\fct(\A,\E)$, dont on notera $\Hh^{(n)}_*(F)$ l'homologie, de sorte que :
\begin{enumerate}
 \item\label{r1} pour tout $i$, $D^{(n)}_i(F)(a)$ est facteur direct naturel de $F(a^{\oplus j})$ pour un certain entier $j$ ($j=(n+1)^i$), $D^{(n)}_0(F)=F$ ; en particulier, $D^{(n)}_*$ commute aux limites, aux colimites, et à la précomposition par un foncteur additif ;
\item si $F$ est projectif, chaque foncteur $D^{(n)}_i(F)$ est projectif ;
\item $\Hh^{(n)}_i(F)$ appartient à $\Pol_n(\A,\E)$ pour tout $i$ et  s'identifie à $q_n(F)$ pour $i=0$ ;
\item si $F$ est polynomial de degré au plus $n$, on a $D^{(n)}_i(F)=0$ pour $i>0$ ;
\item\label{r5} on a $\Hh^{(n)}_j(F)=0$ pour $j<i$ si et seulement si  ${\rm Ext}^j_{\fct(\A,\E)}(F,T)=0$ pour tout $j<i$ et tout foncteur $T$ de $\Pol_n(\A,\E)$.
\end{enumerate}

Pour $n=1$, le complexe $D^{(1)}_*(F)$ est (un modèle de) la stabilisation à la Dold-Puppe de $F$ (cf. \cite{DP} et \cite{JMC}) ; en particulier, 
on peut choisir ce complexe de sorte d'obtenir exactement la construction cubique de Mac Lane (\cite{HML}) lorsque $F=\mathbb{Z}[-]$, dont l'homologie est naturellement isomorphe à l'homologie stable des espaces d'Eilenberg-Mac Lane.

Les complexes $D^{(n)}_*$ peuvent être choisis comme les objets simpliciaux associés à la comonade provenant de l'adjonction entre la diagonale itérée et les effets croisés d'ordre $n+1$, allant de $\fct(\A,\E)$ vers la sous-catégorie pleine de $\fct(\A^{n+1},\E)$ constituée des foncteurs multiréduits (i.e. nuls dès qu'un des $n+1$ arguments est nul). On renvoie le lecteur à \cite{JM-cotr} pour les détails. La propriété~\ref{r5}. mentionnée ci-dessus n'étant pas établie explicitement dans cet article, nous la démontrons rapidement ci-dessous. Une autre référence utile, donnant un point de vue légèrement différent sur ces complexes (en termes de $\Gamma$-modules\,\footnote{La construction $D^{(n)}_*$, dont nous n'avons besoin ici que pour une catégorie source $\A$ additive, s'étend sans aucun changement en supposant seulement que $\A$ est pointée et possède des coproduits finis. L'intervention de la catégorie $\Gamma$ des ensembles finis pointés, qui possède une propriété universelle parmi ces catégories, n'est donc pas étonnante ; la mise en évidence de ses liens avec la construction cubique de Mac Lane est due à Pirashvili \cite{P96}.}), est l'article \cite{Ric} (§\,4) de Richter.

Une conséquence aisée de la propriété~\ref{r1} est la suivante : on dispose d'isomorphismes de complexes $q_d(D^{(n)}_*)\simeq D^{(n)}_*\circ q_d$.

\begin{proof}[Démonstration de la propriété~\ref{r5}]
On raisonne par récurrence sur $i$. La propriété est immédiate pour $i\leq 1$, puisque $\Hh^{(n)}_0\simeq q_0$. On suppose donc $i>1$ et l'assertion établie jusqu'à $i-1$.

On note $\kappa$ le noyau de la différentielle $D^{(n)}_1\to D^{(n)}_0={\rm Id}$. Si $q_n(F)=0$, on a donc une suite exacte courte
$$0\to\kappa(F)\to D^{(n)}_1(F)\to F\to 0.$$
Appliquant le foncteur homologique $\Hh^{(n)}_*$, on obtient une suite exacte longue dans laquelle les morphismes
$$\Hh^{(n)}_i(D^{(n)}_1(F))\simeq D^{(n)}_1(\Hh^{(n)}_i(F))\to\Hh^{(n)}_i(F)$$
(l'isomorphisme résulte de la commutation de $\Hh^{(n)}_i$ à la précomposition par un foncteur additif et de la propriété~\ref{r1}.) sont nuls car $\Hh^{(n)}_i(F)$ appartient à $\Pol_n(\A,\E)$, donc des suites exactes courtes
$$0\to\Hh^{(n)}_{j+1}(F)\to\Hh^{(n)}_j(\kappa(F))\to\Hh^{(n)}_j(D^{(n)}_1(F))\to 0.$$

De façon analogue, si $T$ appartient à $\Pol_n(\A,\E)$, lorsque $q_n(F)=0$, on dispose de suites exactes courtes
$$0\to {\rm Ext}^j(D^{(n)}_1(F),T) \to {\rm Ext}^j(\kappa(F),T)\to {\rm Ext}^{j+1}(F,T)\to 0.$$

On conclut alors aussitôt en utilisant l'hypothèse de récurrence et l'observation que $\Hh^{(n)}_j(F)=0$ entraîne $\Hh^{(n)}_j(D^{(n)}_1(F))\simeq D^{(n)}_1(\Hh^{(n)}_j(F))=0$, et que ${\rm Ext}^j(F,T)=0$ pour tout $T$ dans $\Pol_n(\A,\E)$ entraîne ${\rm Ext}^j(D^{(n)}_1(F),T)=0$ pour tout $T$ dans $\Pol_n(\A,\E)$ (en effet, la propriété~\ref{r1}. montre que $D^{(n)}_1$ est adjoint à gauche à un foncteur exact qui préserve le degré polynomial).
\end{proof}

\begin{rem}
 Cette démonstration constitue une variante de la partie formelle des arguments de \cite{Dja-JKT}, §\,2, qui est due à Scorichenko.
\end{rem}

\smallskip

On utilisera aussi la variante duale de $D^{(n)}_*$, qui est un complexe de cochaînes $D_{(n)}^*$, dont on note $\Hh_{(n)}^*$ la cohomologie, et qui vérifie les propriétés suivantes :
\begin{enumerate}
 \item\label{r21} pour tout $i$, $D_{(n)}^i(F)(a)$ est facteur direct naturel de $F(a^{\oplus j})$ pour un certain entier $j$ ($j=(n+1)^i$), $D^{(n)}_0(F)=F$ ; en particulier, $D^{(n)}_*$ commute aux limites, aux colimites, et à la précomposition par un foncteur additif ;
\item si $F$ est injectif, chaque foncteur $D_{(n)}^i(F)$ est injectif ;
\item $\Hh_{(n)}^i(F)$ appartient à $\Pol_n(\A,\E)$ pour tout $i$ et s'identifie à $p_n(F)$ pour $i=0$ ;
\item si $F$ est polynomial de degré au plus $n$, on a $D_{(n)}^i(F)=0$ pour $i>0$ ;
\item\label{r25} on a $\Hh_{(n)}^j(F)=0$ pour $j<i$ si et seulement si  ${\rm Ext}^j_{\fct(\A,\E)}(T,F)=0$ pour tout $j<i$ et tout foncteur $T$ de $\Pol_n(\A,\E)$ ;
\item on dispose d'isomorphismes de complexes $p_d(D_{(n)}^*)\simeq D_{(n)}^*\circ p_d$.
\end{enumerate}

Un corollaire utile des propriétés~\ref{r25}. et ~\ref{r21}. est le résultat suivant.

\begin{pr}\label{pcol}
Pour tous $n\in\mathbb{N}$ et $i\in\mathbb{N}\cup\{\infty\}$, la classe des foncteurs $F$ de $\fct(\A,\E)$ tels que ${\rm Ext}^j_{\fct(\A,\E)}(T,F)=0$ pour $T$ dans $\Pol_n(\A,\E)$ et $j<i$ est stable par colimites filtrantes.
\end{pr}

Ce résultat est remarquable dans la mesure où l'on ne peut pas toujours engendrer la catégorie $\Pol_n(\A,\E)$ par des foncteurs ayant une résolution projective de type fini dans $\fct(\A,\E)$ (propriété équivalente, pour un foncteur $T$, à la commutation de ${\rm Ext}^*_{\fct(\A,\E)}(T,-)$ aux colimites filtrantes) --- cf. remarque~\ref{rqf}.\,\ref{rfff} ci-après.

\subsection{Rôle de la torsion $p$-primaire bornée}

Soient $\mathfrak{A}$ une sous-catégorie pleine de $\mathbf{Ab}$ stable par quotients et $\Phi$ un foncteur contravariant de $\mathfrak{A}$ vers les $\mathbb{F}_p$-espaces vectoriels.

\begin{defi}\label{dfst}
 \begin{enumerate}
\item Nous dirons que $\Phi$ est {\em stationnaire} sur un objet $V$ de $\mathfrak{A}$ si le morphisme canonique
$$\underset{n\in\mathbb{N}}{\col}\Phi(V/p^n)\to\Phi(V)$$
est un isomorphisme.
\item Nous dirons que $\Phi$ est {\em stationnaire} s'il est sur tout objet de $\mathfrak{A}$.
\item Nous dirons que $\Phi$ est {\em fortement stationnaire} si pour toute famille $(V_t)_{t\in\E}$ d'objets de $V$, le morphisme canonique
$$\underset{n\in\mathbb{N}}{\col}\prod_{t\in\E}\Phi(V_t/p^n)\to\prod_{t\in\E}\Phi(V_t)$$
est un isomorphisme.
 \end{enumerate}
\end{defi}

Le foncteur $\mathbf{Ab}^{op}\to\mathbb{F}_p-\mathbf{Mod}$ qui apparaîtra à la fin de cette article et dont la stationnarité nous préoccupera est le suivant : 
$${\rm Ext}^1_\mathbb{Z}(-,\mathbb{Z}/p)\simeq {\rm Hom}_\mathbb{Z}(-,\mathbb{Z}/p^\infty)\otimes\mathbb{F}_p\simeq {\rm Hom}_\mathbb{Z}(\mathbb{Z}/p,-)^\vee,$$
où $\mathbb{Z}/p^\infty$ désigne la colimite sur $i\in\mathbb{N}$ des $\mathbb{Z}/p^i$, (on rappelle que l'exposant $\vee$ indique la dualité des $\mathbb{F}_p$-espaces vectoriels), et le dernier isomorphisme est induit par l'application bilinéaire de composition
$${\rm Hom}_\mathbb{Z}(\mathbb{Z}/p,-)\times {\rm Hom}_\mathbb{Z}(-,\mathbb{Z}/p^\infty)\to {\rm Hom}_\mathbb{Z}(\mathbb{Z}/p,\mathbb{Z}/p^\infty)\simeq\mathbb{F}_p.$$

Pour établir la nécessité de conditions de torsion bornée dans nos énoncés, nous aurons besoin du résultat élémentaire suivant.

\begin{lm}\label{lmelc}
 Soit $V$ un groupe abélien. Les conditions suivantes sont équivalentes :
\begin{enumerate}
 \item la torsion $p$-primaire de $V$ est bornée ;
\item le foncteur ${\rm Ext}^1_\mathbb{Z}(-,\mathbb{Z}/p)$ est stationnaire sur $V$.
\end{enumerate}

\end{lm}

\begin{proof}
 Si la torsion $p$-primaire de $V$ est bornée par $p^n$, le morphisme canonique ${\rm Hom}_\mathbb{Z}(V,\mathbb{Z}/p^i)\to {\rm Hom}_\mathbb{Z}(V,\mathbb{Z}/p^\infty)$ est un isomorphisme pour $i\geq n$. Comme
$${\rm Ext}^1_\mathbb{Z}(V/p^i,\mathbb{Z}/p)\simeq{\rm Hom}_\mathbb{Z}(V/p^i,\mathbb{Z}/p^i)\otimes\mathbb{F}_p\simeq{\rm Hom}_\mathbb{Z}(V,\mathbb{Z}/p^i)\otimes\mathbb{F}_p,$$
il s'ensuit que le morphisme canonique
$$\underset{i\in\mathbb{N}}{\col} {\rm Ext}^1_\mathbb{Z}(V/p^i,\mathbb{Z}/p)\to {\rm Ext}^1_\mathbb{Z}(V,\mathbb{Z}/p)$$
est un isomorphisme.

Pour la réciproque, on utilise l'isomorphisme de foncteurs ${\rm Ext}^1_\mathbb{Z}(-,\mathbb{Z}/p)\simeq {\rm Hom}_\mathbb{Z}(\mathbb{Z}/p,-)^\vee$. Dire que l'application naturelle
$$\underset{n\in\mathbb{N}}{\col}{\rm Hom}_\mathbb{Z}(\mathbb{Z}/p,V/p^n)^\vee\to{\rm Hom}_\mathbb{Z}(\mathbb{Z}/p,V)^\vee$$
est surjective signifie que toute forme linéaire sur le $\mathbb{F}_p$-espace vectoriel ${\rm Hom}_\mathbb{Z}(\mathbb{Z}/p,V)$ se factorise par l'application canonique ${\rm Hom}_\mathbb{Z}(\mathbb{Z}/p,V)\to{\rm Hom}_\mathbb{Z}(\mathbb{Z}/p,V/p^n)$ pour $n$ assez grand. Cela implique que cette application est injective pour $n$ assez grand (notant $K_n$ son noyau, $(K_n)$ est une suite décroissante de sous-espaces vectoriels de ${\rm Hom}_\mathbb{Z}(\mathbb{Z}/p,V)$, si elle ne stationne pas en $0$ on peut choisir une base de celui-ci qui contient un élément de $K_n$ pour tout $n$, il suffit de choisir la forme linéaire envoyant tous les éléments de cette base sur $1$ pour contredire notre hypothèse), c'est-à-dire que $p^n.V$ est sans $p$-torsion, ou encore que $p^{n+1}.v=0$ implique $p^n.v=0$ pour tout $v\in V$, i.e. que la torsion $p$-primaire de $V$ est bornée par $p^n$. 
\end{proof}

\section{Démonstrations}\label{sd}

\subsection{Le cas rationnel}

On démontre ici le théorème~\ref{th-fol}. On s'appuiera pour cela sur le résultat crucial (spécifique à la caractéristique nulle) suivante :

\begin{pr}\label{pz}
Pour tous entiers $n,i>0$ et tout $d\in\mathbb{N}\cup\{\infty\}$, les foncteurs $\Hh_{(n)}^i(F)$ sont nuls lorsque $F$ est l'un des foncteurs suivants $\mathbf{Ab}^{op}\to\mathbb{Q}-\mathbf{Mod}$ :
\begin{enumerate}
 \item $Fct(-,\mathbb{Q})$ ;
\item $p_d(Fct(-,\mathbb{Q}))$ (c'est-à-dire $Fct_d(-,\mathbb{Q})$ si $d$ est fini et $Fct(-,\mathbb{Q})$ pour $d=\infty$).
\end{enumerate}
\end{pr}

(Pour $n=1$, dans le premier cas, la proposition dit exactement que l'homologie rationnelle stable des espaces d'Eilenberg-Mac Lane est nulle en degrés strictement positifs.)

\begin{proof}
 On commence par la deuxième assertion. On note qu'on peut se limiter au cas où $d$ est fini, puisque $\Hh_{(n)}^i$ commute aux colimites filtrantes.

Le foncteur $p_d(Fct(-,\mathbb{Q}))$ est dual de $q_d(\mathbb{Q}[-])$. Il est classique (cf. \cite{Passi} par exemple) que le gradué de la $\mathbb{Q}$-algèbre d'un groupe abélien $V$ est naturellement isomorphe à l'algèbre symétrique sur $\mathbb{Q}\otimes V$. Or l'algèbre symétrique sur un $\mathbb{Q}$-espace vectoriel est facteur direct naturel de son algèbre tensorielle $T^*$, il suffit donc de vérifier que $\Hh^{(n)}_i(T^d(\mathbb{Q}\otimes -))=0$ (pour $i>0$). Pour $d>n$, cela résulte du lemme~\ref{lmpira} (et de la caractérisation de l'annulation des foncteurs $\Hh^{(n)}_i$ en termes d'annulation cohomologique donnée au paragraphe~\ref{pdp}) --- l'annulation vaut alors même pour $i=0$. Pour $d\leq n$, cela découle de ce que $T^d(\mathbb{Q}\otimes -)$ est de degré inférieur à $n$, de sorte que $D^{(n)}_i(T^d(\mathbb{Q}\otimes -))=0$ pour $i>0$.

Reste à prouver la première assertion, qui équivaut par dualité à l'annulation de $\Hh^{(n)}_i(\mathbb{Q}[-])$ pour $i>0$, où à celle de $\Hh^{(n)}_i(\bar{\mathbb{Q}}[-])$, la barre désignant le noyau de l'idéal d'augmentation. Utilisant que l'homologie rationnelle d'un groupe abélien est naturellement isomorphe à l'algèbre extérieure sur sa rationalisation, qui est facteur direct naturel de puissances tensorielles de cette abélianisation, on obtient un complexe de foncteurs $\mathbf{Ab}\to\mathbb{Q}-\mathbf{Mod}$ de la forme :
$$\dots\to\bar{\mathbb{Q}}[-]^{\otimes j+1}\to\bar{\mathbb{Q}}[-]^{\otimes j}\to\dots\to\bar{\mathbb{Q}}[-]^{\otimes 2}\to\bar{\mathbb{Q}}[-]$$
dont l'homologie est en chaque degré facteur direct d'une puissance tensorielle de la rationalisation. En prenant des produits tensoriels itérés de ce complexe et en utilisant des arguments classiques de multicomplexes, on en déduit par récurrence que, pour tout entier naturel $r$, il existe un complexe de chaînes dont le terme de degré $0$ est $\bar{\mathbb{Q}}[-]$, les termes de degrés strictement positifs sont des sommes directes de foncteurs du type $\bar{\mathbb{Q}}[-]^{\otimes N}$ avec $N>r$, et dont l'homologie est en chaque degré facteur direct d'une somme directe de puissances tensorielles de la rationalisation. Choisissant $r=n$, on voit que $\Hh^{(n)}_i$, pour $i>0$, est nul sur l'homologie de ce complexe (annulation établie en première partie de démonstration) ainsi que sur ses termes de degrés strictement positifs, en utilisant le lemme~\ref{lmpira} (comme plus haut). Par conséquence, un argument standard de suite spectrale montre qu'il en est de même pour son terme de degré nul $\bar{\mathbb{Q}}[-]$, ce qui achève la démonstration.
\end{proof}

\begin{rem}
 Une variante de la deuxième partie de la démonstration consiste à se ramener à la première en utilisant une variation autour du théorème de Dold-Thom (pour la généralisation de la stabilisation à la Dold-Puppe introduite précédemment).
\end{rem}

\begin{proof}[Démonstration du théorème~\ref{th-fol}]
 Il suffit d'établir le théorème lorsque $G$ est de la forme
$$G=p_d(Fct(-,k)\circ\A(-,a))\simeq p_d(Fct(-,k))\circ\A(-,a),$$
(où $a$ parcourt ${\rm Ob}\,\A$), puisque ces foncteurs forment une collection de cogénérateurs de la catégorie $\F_d(\A;k)$. On constate également que $p_d(Fct(-,k))\simeq p_d(Fct(-,\mathbb{Q}))\otimes k$, de sorte qu'on peut supposer $k=\mathbb{Q}$ sans perte de généralité.

Le complexe de cochaînes $D^*_{(d)}(Fct(-,\mathbb{Q})\circ\A(-,a))\simeq D^*_{(d)}(Fct(-,\mathbb{Q}))\circ\A(-,a)$ a une cohomologie nulle en degrés strictement positifs d'après la proposition précédente, et isomorphe à $G$ en degré $0$, il constitue donc une résolution injective de $G$ dans la catégorie $\F(\A;\mathbb{Q})$ (en effet, le foncteur $Fct(-,\mathbb{Q})\circ\A(-,a)$ est injectif).

Utilisant l'isomorphisme $p_d(D^*_{(d)})\simeq D^*_{(d)}\simeq p_d$ et la deuxième partie de la proposition~\ref{pz}, on voit que $p_d\big(D^*_{(d)}(Fct(-,\mathbb{Q})\circ\A(-,a))\big)$ constitue une corésolution de $G$ ; elle est constituée de foncteurs qui sont injectifs dans $\F_d(\A;\mathbb{Q})$.

Comme $p_d$ est adjoint à droite à $i_d$, ces deux observations entraînent le théorème~\ref{th-fol}.
\end{proof}

\begin{rem}
 Une variante de la démonstration précédente consiste à utiliser non pas les complexes $D^{(n)}_*$ ou $D_{(n)}^*$ mais la construction bar sur $\mathbb{Q}[-]$ (ou plutôt sa duale).
\end{rem}

\subsection{Cas de la torsion bornée}

\begin{proof}[Démonstration du théorème~\ref{th1}]
 Les suites spectrales (\ref{gr1}) et (\ref{gr2}) montrent qu'il suffit d'établir que $\mathbf{R}^j p_n(F)=0$ et $\mathbf{L}_j q_n(F)=0$ pour $0<j\leq\frac{n-d+1}{p^r}$, lorsque $F$ appartient à $\F_d(\A;\mathbb{F}_p)$ ; on se contentera de montrer la première assertion, la seconde étant entièrement analogue. On utilise pour cela la formule (\ref{eqa1}).

On vérifie d'abord que le morphisme canonique $\mathbf{R}^j p_n(F)\to\mathbf{R}^j p_m(F)$ (pour $m>n$) est un isomorphisme (sous la condition $0<j\leq\frac{n-d+1}{p^r}$). Cela équivaut à l'annulation de ${\rm Ext}^i_{\F(\A;\mathbb{F}_p)}(S^m_{(p)}\circ\A(a,-),F)$ pour tout objet $a$ de $\A$ et tous entiers $i\leq\frac{n-d+1}{p^r}$ et $m>n$. Il suffit pour cela de montrer que $\Hh^{(d)}_i(S(m))=0$ pour $i\leq\frac{m-d+1}{p^r}$, où $S(m)$ désigne la restriction de $S^m_{(p)}$ à la sous-catégorie pleine de $\mathbf{Ab}$ des groupes abéliens dont la torsion $p$-primaire est bornée par $p^r$.

Cela se montre aussitôt par récurrence sur $d$ en utilisant :
\begin{enumerate}
 \item les suites exactes
$$S(m-j-1)\otimes\Lambda^{j+1}\to S(m-j)\otimes\Lambda^j\to\dots\to S(m-1)\otimes\Lambda^1\to S(m)\to 0$$
où $j=[\frac{n}{p^r}]$ (déduite de la proposition~\ref{kosz})
et
$$0\to\Gamma^m\to\Gamma^{m-1}\otimes\Lambda^1\to\dots\to\Gamma^1\otimes\Lambda^{m-1}\to\Lambda^m\to 0$$
(duale de la suite exacte de Koszul classique ; $\Gamma^j$ est ici le foncteur dual de $\mathbb{F}_p\otimes\Lambda^j$) ;
\item les suites spectrales d'hypercohomologie associées ;
\item le lemme~\ref{pira-gl}.
\end{enumerate}
(Pour $d=1$, la première des suites exactes ci-dessus et le lemme~\ref{lmpira} suffisent. La deuxième suite exacte et le lemme~\ref{pira-gl} permettent de montrer par récurrence sur $d$ que ${\rm Ext}^j(\Lambda^i,F)=0$ si $j<i-d$ et que $F$ est polynomial de degré au plus $d$. Utilisant ce résultat et, de nouveau, la première suite exacte et le lemme~\ref{pira-gl}, on obtient la conclusion recherchée par récurrence sur $d$.)

L'annulation de $\mathbf{R}^j p_n(F)$ pour  $0<j\leq\frac{n-d+1}{p^r}$ qu'on cherche à montrer (pour tout $F$ dans $\F_d(\A;\mathbb{F}_p)$) équivaut à dire que le noyau $N_n$ de la projection $\mathbb{F}_p[-]\twoheadrightarrow Q_n$ vérifie $\Hh_i^{(n)}(N_n\circ\A(a,-))=0$ pour tout $a\in {\rm Ob}\,\A$ et tout entier $i\leq\frac{n-d+1}{p^r}$. Puisque les foncteurs $\Hh_i^{(n)}$ sont stables par précomposition par un foncteur additif et commutent aux colimites filtrantes, il suffit de montrer que $\Hh_i^{(n)}(A_n)=0$, où $A_n$ désigne la restriction de $N_n$ à la sous-catégorie pleine $\mathfrak{A}$ de $\mathbf{Ab}$ des groupes abéliens {\em de type fini} dont la torsion $p$-primaire est bornée par $p^r$.

 En dualisant, on voit que ceci équivaut encore à l'annulation de $\Hh^i_{(n)}$ sur la restriction à $\mathfrak{A}^{op}$ de $Fct(-,\mathbb{F}_p)/Pol_n(-,\mathbb{F}_p) : \mathbf{Ab}^{op}\to\mathbb{F}_p-\mathbf{Mod}$, pour $i\leq\frac{n-d+1}{p^r}$. Mais la première partie de la démonstration montre que l'inclusion $\Pol_n\hookrightarrow\Pol_m$ pour $m\geq n$ induit un isomorphisme $\Hh^i_{(n)}(Pol_n(-,\mathbb{F}_p))\xrightarrow{\simeq}\Hh^i_{(n)}(Pol_m(-,\mathbb{F}_p))$ pour $i\leq\frac{n-d+1}{p^r}$ (la source des foncteurs étant $\mathfrak{A}^{op}$), de sorte qu'il suffit pour conclure d'établir que $\Hh^*_{(n)}(R)=0$, où $R$ désigne la restriction à $\mathfrak{A}^{op}$ du foncteur $Fct(-,\mathbb{F}_p)/Pol(-,\mathbb{F}_p) : \mathbf{Ab}^{op}\to\mathbb{F}_p-\mathbf{Mod}$.

On utilise pour cela le monomorphisme naturel (\ref{eqp}) qui est un isomorphisme sur les groupes abéliens de type fini. Comme $\Hh^i_{(n)}(Fct(-,\mathbb{F}_p))$ est le dual de $\Hh_i^{(n)}(\mathbb{F}_p[-])$, que ce foncteur est polynomial et que $\Hh^i_{(n)}$ commute à la précomposition par un foncteur additif, la conclusion découle donc des lemmes~\ref{lmf} et~\ref{lmf2} ci-dessous.
\end{proof}

\begin{rem}\begin{enumerate}
 \item En utilisant, dans la première partie de la démonstration, la suite exacte de Koszul classique et sa duale, on retrouve, dans le cas où les groupes abéliens de morphismes n'ont pas de $p$-torsion dans $\A$, la borne de Pirashvili (théorème~\ref{th-pira}).
\item Il existe plusieurs variantes de l'argument final, ne reposant pas sur les lemmes~\ref{lmf} et~\ref{lmf2}. L'une d'entre elles consiste à utiliser, lorsque $A : \A\to\mathbf{Ab}$ est un foncteur additif, une résolution projective de $\mathbb{F}_p[A]$ (dans $\F(\A;\mathbb{F}_p)$) à partir d'une résolution projective de $A$ dans la catégorie $\add(\A,\mathbf{Ab})$ et de la résolution bar. On déduit de cela aisément un isomorphisme
$${\rm Ext}^*_{\F(A;\mathbb{F}_p)}(\mathbb{F}_p[A],B)\simeq {\rm Ext}^*_{\add(\A,\mathbf{Ab})}(A,B)$$
lorsque $B : \A\to\mathbb{F}_p-\mathbf{Mod}$ est un autre foncteur additif (on désigne, par abus, encore par $B$ sa postcomposition par l'inclusion $\mathbb{F}_p-\mathbf{Mod}\to\mathbf{Ab}$). On notera que cet isomorphisme est relié de près au critère d'annulation cohomologique de Troesch (\cite{Tefc}, théorème~4). À partir de là, on obtient facilement le résultat pour $n=1$ ; le cas général s'en déduit par un jeu sur la structure des foncteurs polynomiaux (analogue à celui utilisé pour démontrer la dernière assertion du lemme~\ref{lmf}).
\end{enumerate}
\end{rem}

\begin{lm}\label{lmf}
 \begin{enumerate}
  \item La catégorie des foncteurs polynomiaux $\mathfrak{A}\to\mathbb{F}_p-\mathbf{Mod}$ est localement noethérienne.
\item Un foncteur polynomial $\mathfrak{A}\to\mathbb{F}_p-\mathbf{Mod}$ prenant des valeurs de dimension finie est noethérien.
\item Si $F : \mathfrak{A}\to\mathbb{F}_p-\mathbf{Mod}$ est un foncteur polynomial noethérien, alors le foncteur $F^\vee : \mathfrak{A}^{op}\to\mathbb{F}_p-\mathbf{Mod}$ (on rappelle que l'exposant $\vee$ indique la dualité des $\mathbb{F}_p$-espaces vectoriels) est fortement stationnaire (cf. définition~\ref{dfst}).
 \end{enumerate}

\end{lm}

(Dans ce paragraphe, on ne sera servira que du caractère stationnaire ; le caractère {\em fortement} stationnaire sera utilisé dans le paragraphe suivant.)

\begin{proof}
 Les deux premières assertions sont aisées (utiliser par exemple \cite{Dja-pol}, §\,4.3) à partir de l'observation que l'inclusion de la sous-catégorie pleine $\mathfrak{B}$ de $\mathfrak{A}$ des groupes abéliens dont toute la torsion est $p$-primaire (dont tous les objets sont sommes directes d'un nombre fini de copies d'objets en nombre fini, à savoir $\mathbb{Z}$ et $\mathbb{Z}/p^i$ pour $i\leq r$) induit une équivalence des foncteurs polynomiaux $\mathfrak{A}\to\mathbb{F}_p-\mathbf{Mod}$ vers les foncteurs polynomiaux $\mathfrak{B}\to\mathbb{F}_p-\mathbf{Mod}$ (cela découle de ce qu'une fonction polynomiale d'un groupe abélien de torsion, mais sans torsion $p$-primaire, vers $\mathbb{F}_p$ est constante).

On vérifie d'abord la troisième assertion lorsque $F$ est un foncteur additif libre $A_E:={\rm Hom}_\mathbb{Z}(E,-)\otimes\mathbb{F}_p$, puis un produit tensoriel de tels foncteurs, puis une somme directe finie de tels produits tensoriels. Si $F$ est un foncteur polynomial de degré $d$ noethérien, on peut trouver un morphisme d'une telle somme $S$ vers $F$ dont le conoyau $Q$ est de degré $<d$ (en effet, l'effet croisé $cr_d(F)(E_1,\dots,E_d)$ est naturellement isomorphe aux morphismes de $A_{E_1}\otimes\dots A_{E_d}$ vers $F$). Cela permet d'obtenir, par récurrence sur $d$, le résultat souhaité : le fait que la conclusion soit connue pour $F$ et $Q$ montre d'abord que le morphisme qu'on cherche à étudier est injectif pour $F$, foncteur polynomial noethérien de degré au plus $d$ quelconque, ensuite on utilise ce résultat pour le noyau $G$ du morphisme $S\to F$, qui est également polynomial noethérien de degré au plus $d$, ce qui permet de conclure par le lemme des cinq.
\end{proof}

\begin{lm}\label{lmf2}
  Pour tous entiers naturels $i$ et $n$, le foncteur $\Hh^{(n)}_i(\mathbf{F}_p[-])$ prend des valeurs de dimension finie sur les groupes abéliens de type fini.
\end{lm}

\begin{proof}
De façon tout à fait analogue à l'argument final utilisé dans la démonstration de la proposition~\ref{pz}, on voit qu'il existe un complexe de chaînes de foncteurs $\mathbf{Ab}\to\mathbb{F}_p-\mathbf{Mod}$ ayant les propriétés suivantes :
\begin{enumerate}
 \item son terme de degré $0$ est $\bar{\mathbb{F}}_p[-]$ ;
\item ses termes de degrés strictement positifs sont des sommes directes de foncteurs du type $\bar{\mathbb{F}}_p[-]^{\otimes t}$ avec $t>n$ ;
\item son homologie est en tout degré un foncteur (polynomial) prenant des valeurs de dimension finie sur les groupes abéliens de type fini. 
\end{enumerate}
On en déduit la conclusion en considérant la suite spectrale pour $\Hh^{(n)}_*$ associée : elle converge vers $\Hh^{(n)}_*(\bar{\mathbb{F}}_p[-])$ (parce que les termes de degré strictement positif du complexe ont un $\Hh^{(n)}_*$ nul par le lemme~\ref{lmpira}), et sa deuxième page est donnée par les $\Hh^{(n)}_i$ de l'homologie du complexe, qui prennent des valeurs de dimension finie sur les groupes abéliens de type fini.
\end{proof}

\begin{proof}[Démonstration de la proposition~\ref{pr3}]
 Comme corollaire du théorème~\ref{th1}, on obtient
$${\rm Ext}^*_{\F(\A/p^t;\mathbb{F}_p)}(F,G)\simeq\underset{r\geq d}{\col}{\rm Ext}^*_{\F_r(\A/p^t;\mathbb{F}_p)}(F,G)\;;$$
comme le foncteur d'inclusion $\F_r(\A/p^t;\mathbb{F}_p)\to\F_r(\A;\mathbb{F}_p)$ est une équivalence pour $p^t>r$, on en déduit la proposition.
\end{proof}

\subsection{Dernières démonstrations}

\begin{proof}[Démonstration de l'implication 1.$\Rightarrow$2. du théorème~\ref{th2}]
Il suffit de démontrer l'assertion {\em 2.} lorsque $F$ est polynomial (par un argument de colimite filtrante) et $G$ du type
$$G=p_\infty(I)\quad\text{où}\quad I=\underset{t\in E}{\prod}Fct(\A(-,a_t),\mathbb{F}_p)$$
où $(a_t)_{t\in E}$ est une famille d'objets de $\A$. En effet, tout foncteur analytique possède une corésolution par des foncteurs de ce type.

Comme $I$ est injectif dans $\F(\A;\mathbb{F}_p)$, $G$ est injectif dans $\F_\infty(\A;\mathbb{F}_p)$, de sorte qu'il s'agit de vérifier que ${\rm Ext}^*_{\F(\A;\mathbb{F}_p)}(F,I/G)=0$ (pour tout $F$ polynomial), ou encore que $\Hh^*_{(n)}(I/G)=0$ pour tout $n\in\mathbb{N}$, c'est-à-dire que l'inclusion $G\hookrightarrow I$ induit un isomorphisme $\Hh^*_{(n)}(G)\to\Hh^*_{(n)}(I)$.

Comme les foncteurs $\Hh_{(n)}^*$ commutent aux produits, on a
$$\Hh_{(n)}^*(I)\simeq\underset{t\in E}{\prod}H^*_{(n)}\circ\A(-,a_t)$$
où $H^*_{(n)}:=\Hh_{(n)}^*(Fct(-,\mathbb{F}_p))$.

Par ailleurs, comme
$$G=\underset{d\in\mathbb{N}}{\col}p_d\Big(\underset{t\in E}{\prod}Fct(\A(-,a_t),\mathbb{F}_p)\Big)=\underset{d\in\mathbb{N}}{\col}\Big(\underset{t\in E}{\prod}Pol_d(\A(-,a_t),\mathbb{F}_p)\Big)$$
(car les foncteurs $p_d$ commutent aux produits pour $d$ fini), qui s'identifie encore à
$$\underset{d\in\mathbb{N}}{\col}\Big(\underset{t\in E}{\prod}Pol_d(\A(-,a_t),\mathbb{F}_p)\Big)=\underset{(d,i)\in\mathbb{N}^2}{\col}\Big(\underset{t\in E}{\prod}Pol_d(\A(-,a_t)/p^i,\mathbb{F}_p)\Big).$$

Maintenant, les résultats de la section précédente (voir la fin de la démonstration du théorème~\ref{th1}) montrent que l'inclusion induit un isomorphisme
$$\Hh_{(n)}^m(Pol_d(\A(-,a_t)/p^i,\mathbb{F}_p))\to\Hh_{(n)}^m(Fct(\A(-,a_t)/p^i,\mathbb{F}_p))$$
si $mp^i\leq d-n+1$. En utilisant la commutation de $\Hh_{(n)}^m$ aux produits et aux colimites filtrantes, on en déduit :
$$\Hh_{(n)}^m(G)\simeq\underset{i\in\mathbb{N}}{\col}\Big(\underset{t\in E}{\prod}H_{(n)}^m\circ(\A(-,a_t)/p^i)\Big).$$
Via les isomorphismes présents, notre morphisme $\Hh_{(n)}^m(G)\to\Hh_{(n)}^m(I)$ s'identifie au morphisme canonique
$$\underset{i\in\mathbb{N}}{\col}\Big(\underset{t\in E}{\prod}H_{(n)}^m\circ(\A(-,a_t)/p^i)\Big)\to\underset{t\in E}{\prod}H_{(n)}^m\circ\A(-,a_t).$$

La conclusion provient donc des lemmes~\ref{lmf} et~\ref{lmf2}.
\end{proof}

\begin{lm}\label{lm-eml1}
 Le foncteur $\Hh^{(1)}_1(\mathbb{Z}[-]) : \mathbf{Ab}\to\mathbf{Ab}$ est nul.
\end{lm}

\begin{proof}
 La projection canonique $\bar{\mathbb{Z}}[-]\to {\rm Id}$ induit un isomorphisme en appliquant ${\rm Ext}^i(-,A)$ pour $i\leq 1$ et $A$ additif (car la sous-catégorie des foncteurs additifs est épaisse, et Id est projectif dedans puisque représentable par $\mathbb{Z}$), de sorte que $\Hh_1^{(1)}(\bar{\mathbb{Z}}[-])\to\Hh^{(1)}_1(A)$ est un isomorphisme, mais le but est nul car Id est un foncteur de degré $1$.
\end{proof}

\begin{rem}\begin{enumerate}
            \item Cette assertion est équivalente au fait que l'homologie stable de degré $1$ des espaces d'Eilenberg-Mac Lane est nulle.
\item On a commis l'abus de raisonner sur la catégorie des endofoncteurs des groupes abéliens, qui pourrait poser des problèmes ensemblistes, mais cela ne pose pas de problème, ne serait-ce que parce qu'on peut se restreindre aux foncteurs commutant aux colimites filtrantes (qui forment une sous-catégorie équivalente aux foncteurs des groupes abéliens de type fini vers $\mathbf{Ab}$).
           \end{enumerate}
\end{rem}

Conservant la notation $H^*_{(n)}$ employée dans la démonstration précédant le lemme, on en déduit, en utilisant la suite exacte des coefficients universels :
\begin{lm}\label{lm-eml2}
 Le foncteur $H^1_{(1)} : \mathbf{Ab}^{op}\to\mathbb{F}_p-\mathbf{Mod}$ est isomorphe à ${\rm Ext}^1_\mathbb{Z}(-,\mathbb{Z}/p)$.
\end{lm}

\begin{proof}[Démonstration de la proposition~\ref{pr-ext2}]
 Pour l'isomorphisme du bas, on utilise les suites spectrales d'hypercohomologie convergeant vers
$$\mathbf{H}{\rm Ext}^*_{\F(\A;\mathbb{F}_p)}(A,D^*_{(1)}(I))$$
où $I$ est l'injectif $\prod_i \overline{Fct}(\A(-,a_i),\mathbb{F}_p)$ (la barre signifiant qu'on ne considère que les fonctions nulles en $0$) de $\F(\A;\mathbb{F}_p)$.

Leur aboutissement commun est nul en degré total strictement positif, car l'une des deux suites spectrales part de
$${\rm Ext}^*_{\F(\A;\mathbb{F}_p)}(A,D^*_{(1)}(I))\simeq {\rm Hom}_{\F(\A;\mathbb{F}_p)}(D_*^{(1)}(A),I),$$
et $D_*^{(1)}(A)$ est nul en degré $*>0$ pour $A$ additif.

Quant à l'autre, sa deuxième page est donnée par
$${\rm E}_2^{p,q}={\rm Ext}^p_{\F(\A;\mathbb{F}_p)}(A,\Hh^q_{(1)}(I)).$$
Comme son aboutissement est nul en degré total $1$, la différentielle
$$d_2^{0,1} : {\rm Hom}_{\F(\A;\mathbb{F}_p)}(A,\Hh^1_{(1)}(I))\to{\rm Ext}^2_{\F(\A;\mathbb{F}_p)}(A,\Hh^0_{(1)}(I))$$
est un isomorphisme. Comme $\Hh^0_{(1)}(I)\simeq p_1(I)\simeq\prod_i {\rm Hom}_\mathbb{Z}(\A(-,a_i),\mathbb{F}_p)=T$ et que $\Hh^1_{(1)}(I)\simeq {\rm Ext}^1_\mathbb{Z}(-,\mathbb{Z}/p)\circ T$ grâce au lemme~\ref{lm-eml2}, on obtient l'isomorphisme recherché.

L'isomorphisme
$$\underset{d\in\mathbb{N}^*}{\col}{\rm Ext}^2_{\F_d(\A)}(A,{\rm Hom}_\mathbb{Z}(T,\mathbb{F}_p))\simeq\underset{i\in\mathbb{N}}{\col}{\rm Hom}_{\F(\A)}(A,{\rm Ext}^1_\mathbb{Z}(-,\mathbb{Z}/p)\circ (T/p^i))$$
s'en déduit en utilisant la proposition~\ref{pr3}.

Quant au dernier isomorphisme, on l'obtient en raisonnant comme au début de la démonstration, mais en remplaçant $I$ par $p_\infty(I)$ et en travaillant dans $\F_\infty(\A;\mathbb{F}_p)$ (avec les mêmes arguments que d'habitude pour passer de $Pol(-,\mathbb{F}_p)$ à $\underset{i}{\col}Fct(-,\mathbb{F}_p)\circ (-/p^i)$).
\end{proof}

\begin{proof}[Fin de la démonstration du théorème~\ref{th2}]
 L'implication $2.\Rightarrow 3.$ est évidente et l'implication $3.\Rightarrow 1.$ découle de la proposition~\ref{pr-ext2} et du lemme~\ref{lmelc}. 
\end{proof}

\begin{rem}\label{rqf}
\begin{enumerate}
 \item En raisonnant de façon duale, on obtient un isomorphisme naturel
\begin{equation}\label{ex2}
 {\rm Ext}^2_{\F(\A;\mathbb{F}_p)}(\A(a,-)\otimes\mathbb{F}_p,A)\simeq {\rm Hom}_{\F(\A;\mathbb{F}_p)}({\rm Hom}_\mathbb{Z}(\mathbb{Z}/p,-)\circ\A(a,-),A)
\end{equation}
pour tous $a\in {\rm Ob}\,\A$ et $A\in {\rm Ob}\,\F_1(\A;\mathbb{F}_p)$, dont on déduit
$$\underset{d\in\mathbb{N}^*}{\col}{\rm Ext}^2_{\F_d(\A;\mathbb{F}_p)}(\A(a,-)\otimes\mathbb{F}_p,A)\simeq\underset{i\in\mathbb{N}}{\col}{\rm Hom}_{\F(\A;\mathbb{F}_p)}({\rm Hom}_\mathbb{Z}(\mathbb{Z}/p,-)\circ (\A(a,-)/p^i),A).$$
\item En utilisant ce qui précède, ou les isomorphismes de la proposition~\ref{pr-ext2}, et le lemme~\ref{lmelc}, on voit facilement que, si le morphisme canonique
$$\underset{d\in\mathbb{N}^*}{\col}{\rm Ext}^2_{\F_d(\A;\mathbb{F}_p)}(A,B)\to{\rm Ext}^2_{\F(\A;\mathbb{F}_p)}(A,B)$$
est un isomorphisme pour tous foncteurs additifs $A$ et $B$, alors il existe un entier $r$ tel que la torsion de tous les groupes abéliens $\A(a,b)$ soit bornée par $p^r$ (prendre pour $A$ une grosse somme directe de foncteurs additifs libres, et pour $B$ un gros produit direct de foncteurs additifs colibres, comme dans la proposition~\ref{pr-ext2}). En revanche, cette condition est vérifiée sous la première condition (plus faible) du théorème~\ref{th2}, si l'on se restreint aux foncteurs $A$ (additifs) de type fini.
\item Si l'on remplace l'anneau de coefficients $\mathbb{F}_p$ par $\mathbb{Z}$ au but, il ne se passe rien sur le ${\rm Ext}^2$ entre foncteurs additifs : les morphismes naturels
$${\rm Ext}^2_{\F_1(\A;\mathbb{Z})}(A,B)\to {\rm Ext}^2_{\F_d(\A;\mathbb{Z})}(A,B)\to {\rm Ext}^2_{\F(\A;\mathbb{Z})}(A,B)$$
sont des isomorphismes pour tous $d\geq 1$ et tous $A, B\in {\rm Ob}\,\F_1(\A;\mathbb{Z})$ (utiliser le lemme~\ref{lm-eml1}). C'est à partir du groupe ${\rm Ext}^3$ que les <<~pathologies~>> observées précédemment lorsque la torsion $p$-primaire n'est pas bornée (pour un certain nombre premier $p$) se manifestent.
\item\label{rfff} L'isomorphisme (\ref{ex2}) montre également que ${\rm Ext}^2_{\F(\A;\mathbb{F}_p)}(\A(a,-)\otimes\mathbb{F}_p,-)$, qui s'identifie à l'évaluation en $a$ de $\mathbf{R}^2 p_1$ (cf. (\ref{eqa1})), commute aux colimites filtrantes de monomorphismes (resp. aux colimites filtrantes) entre foncteurs additifs si et seulement si le foncteur additif ${\rm Hom}_\mathbb{Z}(\mathbb{Z}/p,-)\circ\A(a,-)$ est de type fini (resp. de présentation finie). Ce n'est pas toujours le cas (ce qui illustre la remarque~\ref{rcfm}), comme le montre l'exemple suivant.

Considérons la catégorie préadditive ayant une infinité dénombrable d'objets $a, t_0, t_1, \dots,t_n,\dots$ avec
$${\rm End}\,(t_n)=\mathbb{Z}/p,\; {\rm Hom}\,(t_n,a)=0\;, {\rm Hom}\,(a,t_n)=\mathbb{Z}/p\;\text{ pour tout }n,$$
$${\rm End}\,(a)=\mathbb{Z}, \text{ et }{\rm Hom}\,(t_i,t_j)=0\text{ si }i\neq j$$
(il existe une et une seule composition bilinéaire). En saturant cette catégorie par des sommes directes formelles finies, on obtient une petite catégorie additive $\A$, dans laquelle tous les groupes abéliens de morphismes sont des groupes abéliens de type fini. Le foncteur ${\rm Hom}_\mathbb{Z}(\mathbb{Z}/p,-)\circ\A(a,-)$ est isomorphe à la somme directe sur les entiers naturels $n$ de $\A(t_n,-)$, il n'est donc pas de type fini.
\end{enumerate}

%
\end{rem}

%

\bibliographystyle{plain}
\bibliography{bibli-polext.bib}

\begin{thebibliography}{10}

\bibitem{Dja-pol}
Aur\'elien Djament.
\newblock Des propri\'et\'es de finitude des foncteurs polynomiaux.
\newblock Pr\'epublication disponible sur
  http://hal.archives-ouvertes.fr/hal-00853071.

\bibitem{Dja-JKT}
Aur\'elien Djament.
\newblock Sur l'homologie des groupes unitaires à coefficients polynomiaux.
\newblock {\em J. K-Theory}, 10(1):87--139, 2012.

\bibitem{DP}
Albrecht Dold and Dieter Puppe.
\newblock Homologie nicht-additiver {F}unktoren. {A}nwendungen.
\newblock {\em Ann. Inst. Fourier Grenoble}, 11:201--312, 1961.

\bibitem{EML}
Samuel Eilenberg and Saunders Mac~Lane.
\newblock On the groups {$H(\Pi,n)$}. {II}. {M}ethods of computation.
\newblock {\em Ann. of Math. (2)}, 60:49--139, 1954.

\bibitem{FFSS}
Vincent Franjou, Eric~M. Friedlander, Alexander Scorichenko, and Andrei Suslin.
\newblock General linear and functor cohomology over finite fields.
\newblock {\em Ann. of Math. (2)}, 150(2):663--728, 1999.

\bibitem{FLS}
Vincent Franjou, Jean Lannes, and Lionel Schwartz.
\newblock Autour de la cohomologie de {M}ac {L}ane des corps finis.
\newblock {\em Invent. Math.}, 115(3):513--538, 1994.

\bibitem{FrS}
Vincent Franjou and Jeffrey~H. Smith.
\newblock A duality for polynomial functors.
\newblock {\em J. Pure Appl. Algebra}, 104(1):33--39, 1995.

\bibitem{HPV}
Manfred Hartl, Teimuraz Pirashvili, and Christine Vespa.
\newblock Polynomial functors from algebras over a set-operad and non-linear
  {M}ackey functors.
\newblock arXiv:1209.1607.

\bibitem{HLS}
Hans-Werner Henn, Jean Lannes, and Lionel Schwartz.
\newblock The categories of unstable modules and unstable algebras over the
  {S}teenrod algebra modulo nilpotent objects.
\newblock {\em Amer. J. Math.}, 115(5):1053--1106, 1993.

\bibitem{JM-cotr}
B.~Johnson and R.~McCarthy.
\newblock Deriving calculus with cotriples.
\newblock {\em Trans. Amer. Math. Soc.}, 356(2):757--803 (electronic), 2004.

\bibitem{JMC}
Brenda Johnson and Randy McCarthy.
\newblock Linearization, {D}old-{P}uppe stabilization, and {M}ac {L}ane's
  {$Q$}-construction.
\newblock {\em Trans. Amer. Math. Soc.}, 350, 1998.

\bibitem{Ku}
Nicholas~J. Kuhn.
\newblock The generic representation theory of finite fields: a survey of basic
  structure.
\newblock In {\em Infinite length modules ({B}ielefeld, 1998)}, Trends Math.,
  pages 193--212. Birkh\"auser, Basel, 2000.

\bibitem{HML}
Saunders Mac~Lane.
\newblock Homologie des anneaux et des modules.
\newblock In {\em Colloque de topologie alg\'ebrique, {L}ouvain, 1956}, pages
  55--80. Georges Thone, Li\`ege; Masson \& Cie, Paris, 1957.

\bibitem{Passi}
Inder Bir~S. Passi.
\newblock {\em Group rings and their augmentation ideals}, volume 715 of {\em
  Lecture Notes in Mathematics}.
\newblock Springer, Berlin, 1979.

\bibitem{P-add}
T.~I. Pirashvili.
\newblock Higher additivizations.
\newblock {\em Trudy Tbiliss. Mat. Inst. Razmadze Akad. Nauk Gruzin. SSR},
  91:44--54, 1988.

\bibitem{P-extor}
Teimuraz Pirashvili.
\newblock Polynomial approximation of {${\rm Ext}$} and {${\rm Tor}$} groups in
  functor categories.
\newblock {\em Comm. Algebra}, 21(5):1705--1719, 1993.

\bibitem{P96}
Teimuraz Pirashvili.
\newblock Kan extension and stable homology of {E}ilenberg-{M}ac {L}ane spaces.
\newblock {\em Topology}, 35(4):883--886, 1996.

\bibitem{PDK}
Teimuraz Pirashvili.
\newblock Dold-{K}an type theorem for {$\Gamma$}-groups.
\newblock {\em Math. Ann.}, 318(2):277--298, 2000.

\bibitem{Ric}
Birgit Richter.
\newblock Taylor towers for {$\Gamma$}-modules.
\newblock {\em Ann. Inst. Fourier (Grenoble)}, 51(4):995--1023, 2001.

\bibitem{Tefc}
Alain Troesch.
\newblock Quelques calculs de cohomologie de compositions de puissances
  sym\'etriques.
\newblock {\em Comm. Algebra}, 30(7):3351--3382, 2002.

\end{thebibliography}
\end{document}